\documentclass[a4paper]{article}
\usepackage{hyperref,authblk,amsmath,amssymb,amsthm,color}
\usepackage{pgf,tikz}
\usepackage{mathrsfs}
\usetikzlibrary{arrows}
\usepackage{comment}

\title{Communicating harmonic pencils of lines}
\author{Norbert Hungerb\"uhler}
\author{Clemens Pohle}
\affil{Department of Mathematics, ETH Z\"urich, 8092 Z\"urich, Switzerland}
\date{}

\newtheorem{theorem}{Theorem}
\newtheorem{proposition}[theorem]{Proposition}
\newtheorem{corollary}[theorem]{Corollary}
\newtheorem{lemma}[theorem]{Lemma}

\theoremstyle{definition}

\definecolor{darkgreen}{rgb}{0,.6,0}

\renewcommand{\theenumi}{\roman{enumi}}
\renewcommand{\labelenumi}{(\theenumi)}

\begin{document}
\maketitle
\begin{abstract}
\noindent Suppose there are $n$ harmonic pencils of lines given in the plane.
We are interested in the question whether certain triples of these lines
are concurrent or if triples of intersection points of these lines are collinear,
provided that we impose suitable conditions on the initial harmonic pencils. Such conditions
can be that certain of the given lines coincide, are concurrent or that certain intersection points
are collinear. The study of these questions for $n=2,%$ 
%(in Section~\ref{sec:introduction}, $n=
3, 4$ %in Sections~\ref{sec:affine-triangles} and~\ref{sec:projective-triangles},
%and $n=4$ in Sections~\ref{sec:affine-quadrilaterals} and~\ref{sec:projective-quadrilaterals} 
sheds light on
some well known affine configurations %(Sections~\ref{sec:affine-triangles} and~\ref{sec:affine-quadrilaterals}) 
and provides new results in the projective setting. %(Sections~\ref{sec:projective-triangles} and~\ref{sec:projective-quadrilaterals}). 
As applications, we will formulate generalizations or stronger versions of
the theorems of Pappus, Desargues, Ceva and Menelaos. Notably, the generalized theorems of Ceva and Menelaos suggest a new way to generalize the terms `collinearity' and `concurrency'.
\end{abstract}
%%%%%%%%%%%%%%%%%%%%%%%%%%%%%%%%%%%%%%%%%
%%%%%%%%%%%%%%%%%%%%%%%%%%%%%%%%%%%%%%%%%
%%%%%%%%%%%%%%%%%%%%%%%%%%%%%%%%%%%%%%%%%
\section{Introduction}\label{sec:introduction}
Before we start, let us briefly fix some notation: Points will be denoted by capital letters, lines
by small letters. The intersection point of two lines $g$ and $h$ is $g\times h$,
and $A\times B$ denotes the line trough $A$ and $B$. The term $\frac{AB}{CD}$
is the quotient of the oriented lengths of the segments $AB$ and $CD$ 
on an oriented line.  The cross-ratio of four collinear points $A,B,C,D$ is
$$
\operatorname{CR}(A,B;C,D):=\frac{AC}{CB}\cdot\frac{BD}{DA}.
$$
A harmonic quadruple of collinear points $A,B,X,Y$ or concurrent lines $g,h,u,v$
will be denoted by
$(AB;XY)$ and $(gh;uv)$, respectively: $Y$ is the harmonic conjugate of $X$
with respect to $A$ and $B$ (see, e.g.,~\cite{coxeter}). Finally, we denote
by $[ABC]$ the signed area of a triangle $ABC$ with respect to a given orientation of the 
plane. In particular $[ABC]=-[CBA]$ (see Figure~\ref{fig-ABC}).

Suppose we are given points $A_1,A_2,\ldots,A_n$ in the projective plane
and each point $A_i$ carries a harmonic pencil of four lines which meet in $A_i$.
In general there will be $8n(n-1)$ further intersection points of these lines.
Interesting constellations concerning these lines and their points of intersection are
\begin{itemize}
\item concurrent triples of lines,
\item collinear triples of intersection points.
\end{itemize}
Of course, without any further conditions no such constellation will occur.
This changes if we let the $n$ initial harmonic pencils of lines
``communicate'' with each other. This means that we impose conditions of the following
kind:
\begin{itemize}
\item certain lines coincide,
\item certain lines are concurrent,
\item certain points of intersection are collinear.
\end{itemize}
Let us have a look at an example with two points $A_1,A_2$. Here we
have the following proposition, illustrated by Figure~\ref{fig:two-pencils}.
\begin{proposition}\label{thm:two-pencils}
Let $A_1,A_2$ be two points in the projective plane and $(a_{i1}a_{i2};b_{i1}b_{i2})$
harmonic pencils of lines at point $A_i$, $i\in\{1,2\}$. Then the following
holds: If three of the points
$$
a_{11}\times a_{21}, \quad a_{12}\times a_{22}, \quad 
b_{11}\times b_{21}, \quad b_{12}\times b_{22}
$$
are collinear, then so are all four points.
\end{proposition}
\begin{proof}
The intersection points $P,Q,X,Y$ of a harmonic pencil with a
line $g$ are harmonic points. And vice versa, if four collinear
harmonic points are connected with a point $A_1$ or $A_2$,
the connecting lines are a harmonic pencil. Therefore the 
claim follows from the fact that the fourth harmonic line
is determined by the three others.
\end{proof}
%This is a special case of Theorem~\ref{thm:projective-quadrilateral} (see Section~\ref{sec:projective-quadrilaterals}).
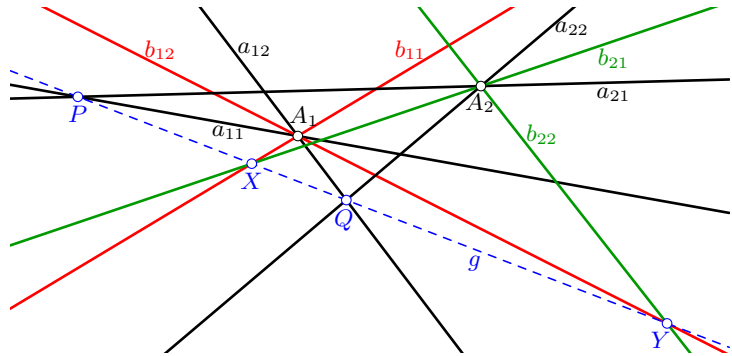
\begin{figure}[h]
\begin{center}
\begin{tikzpicture}[line cap=round,line join=round,x=8,y=8]
\clip(-14.730570790934566,-12.28672021134537) rectangle (18.911531168763908,4);
\draw [line width=1.pt,domain=-14.730570790934566:18.911531168763908] plot(\x,{(-20.26173681642186-1.5833669768678047*\x)/8.8100162559055});
\draw [line width=1.pt,domain=-14.730570790934566:18.911531168763908] plot(\x,{(--16.64513402985136--5.886877221687991*\x)/-4.4253077045792635});
\draw [line width=1.pt,color=red,domain=-14.730570790934566:18.911531168763908] plot(\x,{(--16.07831012920593-7.470244198555797*\x)/-12.34214258891831});
\draw [line width=1.pt,color=red,domain=-14.730570790934566:18.911531168763908] plot(\x,{(-33.566167163890896-6.311102537793815*\x)/12.333444877538447});
\draw [line width=1.pt,domain=-14.730570790934566:18.911531168763908] plot(\x,{(--61.07424400921481-8.769417102652456*\x)/-10.271585773014252});
\draw [line width=1.pt,domain=-14.730570790934566:18.911531168763908] plot(\x,{(-1.455548138168488-0.4465906857832269*\x)/-16.64565283373851});
\draw [line width=1.pt,color=darkgreen,domain=-14.730570790934566:18.911531168763908] plot(\x,{(-38.45021960543335--5.115493309880599*\x)/-3.978717018796036});
\draw [line width=1.pt,color=darkgreen,domain=-14.730570790934566:18.911531168763908] plot(\x,{(--154.7488296342937-23.93164502474728*\x)/-70.12986417502907});
\draw [line width=0.6pt,dash pattern=on 3pt off 3pt,color=blue,domain=-14.730570790934566:18.911531168763908] plot(\x,{(--91.35549122694164--7.534883419740403*\x)/-19.425429045049324});
\begin{small}
\draw [color=black,fill=white] (-1.2702586456067821,-2.0715570431159733) circle (1.8pt);
\draw[color=black] (-.98,-1.15) node {$A_1$};
\draw[color=black] (-4.5,-1.9) node {$a_{11}$};
\draw[color=black] (-3.3,2) node {$a_{12}$};
\draw[color=red] (4.,2) node {$b_{11}$};
\draw[color=red] (-7.7,2) node {$b_{12}$};
\draw [color=black,fill=white] (7.296162690780596,0.28319384555922317) circle (1.8pt);
\draw[color=black] (7.25,-0.5) node {$A_2$};
\draw[color=black] (13.5,-.1) node {$a_{21}$};
\draw[color=black] (11.55,3) node {$a_{22}$};
\draw[color=darkgreen] (10.15,-2) node {$b_{22}$};
\draw[color=darkgreen] (13.5,1.6) node {$b_{21}$};

\draw [color=blue] (7,-7.3)  node[anchor=north] {$g$};

\draw [color=blue] (15.7,-10.90978732575519)  node[anchor=north] {$Y$};
\draw [color=blue,fill=white] (16.00181471291406,-10.90978732575519) circle (1.8pt);

\draw [color=blue] (-3.423614332135262,-3.3749039060147865) node[anchor=north]  {$X$};
\draw [color=blue,fill=white] (-3.423614332135262,-3.3749039060147865) circle (1.8pt);

\draw [color=blue,fill=white] (.88,-5.15) node[anchor=north]  {$Q$};
\draw [color=blue,fill=white] (1.0002565329103499,-5.091967142978203) circle (1.8pt);

\draw [color=blue,fill=white] (-11.557799530162267,-0.2226441652465594) node[anchor=north]  {$P$};
\draw [color=blue,fill=white] (-11.557799530162267,-0.2226441652465594) circle (1.8pt);
\end{small}
\end{tikzpicture}
\caption{The configuration in Proposition~\ref{thm:two-pencils}.}\label{fig:two-pencils}
\end{center}
\end{figure}
A nice special case of Proposition~\ref{thm:two-pencils} occurs, when $a_{11}=a_{21}=a$ (see Figure~\ref{fig:corollary}):
\begin{corollary}\label{cor:two-pencils}
Let $A_1,A_2$ be two points in the projective plane with $a=A_1\times A_2$,
and $(a,a_{i};b_{i1}b_{i2})$
harmonic pencils of lines at point $A_i$, $i\in\{1,2\}$. Then the triples of points
\begin{alignat*}{3}
&a_{1}\times a_{2}, \quad&&b_{11}\times b_{21}, \quad&& b_{12}\times b_{22}\\
&a_{1}\times a_{2}, \quad&&b_{11}\times b_{22}, \quad&& b_{12}\times b_{21}
\end{alignat*}
are collinear.
\end{corollary}
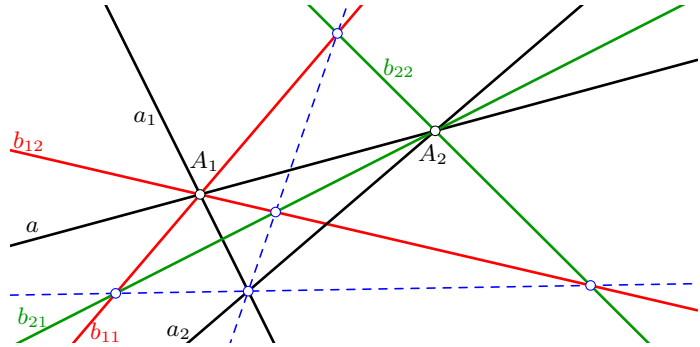
\begin{figure}[h]
\begin{center}
\begin{tikzpicture}[line cap=round,line join=round,x=8,y=8]
\clip(-11.882388126895968,-9) rectangle (20.258588595411787,6.8961286602331695);
\draw [line width=1pt,domain=-11.882388126895968:20.258588595411787] plot(\x,{(--30.185962996371362--7.570180775649525*\x)/-3.7377103347113945});
\draw [line width=1pt,color=red,domain=-11.882388126895968:20.258588595411787] plot(\x,{(-15.214492571772734-11.629544593121711*\x)/-9.837070603796201});
\draw [line width=1pt,color=red,domain=-11.882388126895968:20.258588595411787] plot(\x,{(-36.2619143127954-3.1496593443909777*\x)/13.406468139811231});
\draw [line width=1pt,domain=-11.882388126895968:20.258588595411787] plot(\x,{(--64.9143629745122-9.486223257093233*\x)/-10.975423082233656});
\draw [line width=1pt,domain=-11.882388126895968:20.258588595411787] plot(\x,{(--13.-3.*\x)/-11.});
\draw [line width=1pt,color=darkgreen,domain=-11.882388126895968:20.258588595411787] plot(\x,{(-39.68385307095479--4.4110670294994865*\x)/-4.395316834958898});
\draw [line width=1pt,color=darkgreen,domain=-11.882388126895968:20.258588595411787] plot(\x,{(--47.58630820154577-7.865275287275761*\x)/-15.335894096660317});
\draw [line width=.6pt,dash pattern=on 3pt off 3pt,color=blue,domain=-11.882388126895968:20.258588595411787] plot(\x,{(--145.45075809211988-0.37440719330424166*\x)/-22.21151729317625});
\draw [line width=.6pt,dash pattern=on 3pt off 3pt,color=blue,domain=-11.882388126895968:20.258588595411787] plot(\x,{(-5.6099128633764686--3.7311607272282488*\x)/1.2804214226471884});
\begin{small}
\draw [color=black,fill=white] (-3.,-2.) circle (1.8pt);
\draw[color=black] (-2.8,-0.44) node {$A_1$};
\draw[color=black] (-5.55,1.5923371218985802) node {$a_{1}$};
\draw [fill=red] (-12.837070603796201,-13.629544593121711) circle (1.8pt);
%\draw[color=red] (-11.723274380745929,7.34695094099161) node {$D$};
\draw[color=red] (-7.4,-8.5) node {$b_{11}$};
\draw[color=red] (-11.007262523070757,0.5) node {$b_{12}$};
\draw [color=black,fill=white] (8.,1.) circle (1.8pt);
\draw[color=black] (7.9,-0.1) node {$A_2$};
\draw[color=black] (-4.,-8.5) node {$a_{2}$};
\draw[color=black] (-10.848148776920718,-3.552340670285974) node {$a$};
\draw[color=darkgreen] (6.3,3.9790433141491466) node {$b_{22}$};
\draw[color=darkgreen] (-10.8,-7.8) node {$b_{21}$};
\draw [color=blue,fill=white] (15.265089762501265,-6.291123511923858) circle (1.8pt);
\draw [color=blue,fill=white] (-6.946427530674983,-6.665530705228099) circle (1.8pt);
\draw [color=blue,fill=white] (-0.7480227462902367,-6.561047643235842) circle (1.8pt);
\draw [color=blue,fill=white] (0.5323986763569518,-2.829886916007593) circle (1.8pt);
\draw [color=blue,fill=white] (3.423012054120042,5.5933891413487755) circle (1.8pt);
\end{small}
\end{tikzpicture}
\caption{Collinear intersection points of two communicating harmonic pencils (Corollary~\ref{cor:two-pencils}).}\label{fig:corollary}
\end{center}
\end{figure}
In Sections~\ref{sec:affine-triangles} and~\ref{sec:affine-quadrilaterals} we will illustrate the concept
of communicating harmonic pencils of lines with two examples which live in the affine plane. 
There the harmonic pencils do not communicate in the above described purely projective way. The idea
is then to carry these two configurations into the projective plane and to
replace the affine conditions by projectively invariant conditions. This
will be done in Sections~\ref{sec:projective-triangles} and~\ref{sec:projective-quadrilaterals}, respectively. Finally, in Section~\ref{sec:constellations}, we will show how our results can be used to expand some well-known theorems.
%%%%%%%%%%%%%%%%%%%%%%%%%%%%%%%%%%%%%%%%%
\section{Harmonic pencils in affine triangles}\label{sec:affine-triangles}
Let $A_1A_2A_3$ be an affine triangle with sides $a_1,a_2,a_3$, where
the notation is such that $A_i\notin a_i$. Let $g_i$ denote the inner angle 
bisectors of the triangle and $h_i$ the outer angle bisectors in $A_i$.
Then $(a_ia_{i+1}; g_{i+2}h_{i+2})$ is a harmonic pencil of lines for $i\in\{1,2,3\}$ (indices are taken cyclically). 
Obviously,  the following lines are concurrent.
\begin{enumerate}
\item $g_1,g_2,g_3$\label{i}: The intersection point is the center of the incircle of the triangle),
\item $g_i,h_{i+1},h_{i+2}$ for $i\in\{1,2,3\}$\label{ii}: The intersection points are the centers of the excircles (see Figure~\ref{fig:affine-triangle}).
\end{enumerate}
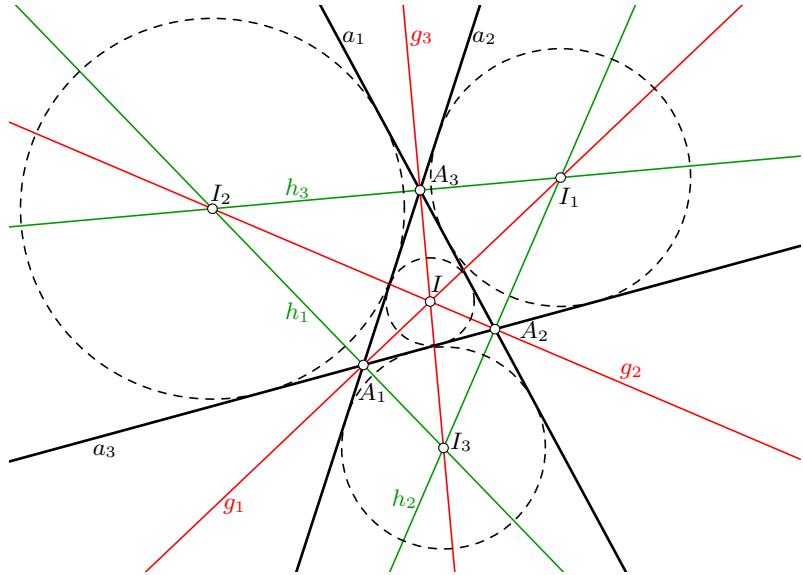
\begin{figure}[h]
\begin{center}
\begin{tikzpicture}[line cap=round,line join=round,x=16,y=16]
\clip(-5.749486941635399,-8.20164253410491) rectangle (12.754592187589372,5.125979015931666);
\draw [line width=1.pt,domain=-5.749486941635399:12.754592187589372] plot(\x,{(-12.431058138197987--0.8456433366034064*\x)/3.077202141529057});
\draw [line width=1.pt,domain=-5.749486941635399:12.754592187589372] plot(\x,{(-14.017661823729167--3.277457012158247*\x)/-1.7521509702542732});
\draw [line width=1.pt,domain=-5.749486941635399:12.754592187589372] plot(\x,{(--14.881627432625653-4.123100348761653*\x)/-1.325051171274784});
\draw [line width=.6pt,color=darkgreen,domain=-5.749486941635399:12.754592187589372] plot(\x,{(-0.48246428175059086-0.7220615187347782*\x)/0.6918288539533645});
\draw [line width=.6pt,color=red,domain=-5.749486941635399:12.754592187589372] plot(\x,{(-4.167674858116997--0.6918288539533645*\x)/0.7220615187347782});
\draw [line width=.6pt,color=darkgreen,domain=-5.749486941635399:12.754592187589372] plot(\x,{(-6.142254005344644--0.918775192629696*\x)/0.3947811360846225});
\draw [line width=.6pt,color=red,domain=-5.749486941635399:12.754592187589372] plot(\x,{(--0.07903021151467993--0.3947811360846225*\x)/-0.918775192629696});
\draw [line width=.6pt,color=darkgreen,domain=-5.749486941635399:12.754592187589372] plot(\x,{(-0.42991034172486625-0.08987895376855971*\x)/-0.9959526965019316});
\draw [line width=.6pt,color=red,domain=-5.749486941635399:12.754592187589372] plot(\x,{(--3.9144829924369327-0.9959526965019316*\x)/0.08987895376855971});
\draw [line width=.6pt,dash pattern=on 3pt off 3pt] (4.097020315042089,-1.846432685228709) circle (1.0292392870658316);
\draw [line width=.6pt,dash pattern=on 3pt off 3pt] (7.147908089166828,1.0767146333148387) circle (3.039450461819348);
\draw [line width=.6pt,dash pattern=on 3pt off 3pt] (-0.9956688822251079,0.3418040490137894) circle (4.4887401430557725);
\draw [line width=.6pt,dash pattern=on 3pt off 3pt] (4.408556991386304,-5.2985844944635865) circle (2.382059135433858);
\begin{small}
\draw [fill=white] (2.534948828725216,-3.3431003487616535) circle (1.8pt);
\draw[color=black] (2.75,-4.) node {$A_1$};
\draw [fill=white] (5.612150970254273,-2.497457012158247) circle (1.8pt);
\draw[color=black] (6.53,-2.55) node {$A_2$};
\draw[color=black] (-3.512601426646834,-5.367473452286939) node {$a_3$};
\draw [fill=white] (3.86,0.78) circle (1.8pt);
\draw[color=black] (4.45,1.13) node {$A_3$};
\draw[color=black] (2.3196969318050242,4.37) node {$a_1$};
\draw[color=black] (5.364671978386316,4.37) node {$a_2$};
\draw[color=darkgreen] (1,-2.1) node {$h_1$};
\draw[color=red] (-0.48,-6.679155010814265) node {$g_1$};
\draw[color=darkgreen] (3.5,-6.5) node {$h_2$};
\draw[color=red] (8.8,-3.5) node {$g_2$};
\draw[color=darkgreen] (1,.83) node {$h_3$};
\draw[color=red] (3.9,4.37) node {$g_3$};
\draw [fill=white] (4.097020315042089,-1.846432685228709) circle (1.8pt);
\draw[color=black] (4.25,-1.44) node {$I$};
\draw [fill=white] (-0.9956688822251079,0.3418040490137894) circle (1.8pt);
\draw[color=black] (-0.7955467696973741,0.6756308709282398) node {$I_2$};
\draw [fill=white] (7.147908089166828,1.0767146333148387) circle (1.8pt);
\draw[color=black] (7.355617201151006,0.605362216007133) node {$I_1$};
\draw [fill=white] (4.408556991386304,-5.2985844944635865) circle (1.8pt);
\draw[color=black] (4.82,-5.1566674875236185) node {$I_3$};
\end{small}
\end{tikzpicture}
\caption{Incircle and excircles in triangle $A_1A_2A_3$.}\label{fig:affine-triangle}
\end{center}
\end{figure}
In Section~\ref{sec:projective-triangles} we will replace the 
angle bisectors in the triangle by lines $g_i, h_i$ which form harmonic pencils together
with the sides of the triangle. It will turn out that in this purely projective setting
(\ref{i}) occurs if and only if one of the concurrencies in (\ref{ii}) occurs (see Theorem~\ref{thm:projective-triangle}).

The configuration in the next section is less well-known.
%%%%%%%%%%%%%%%%%%%%%%%%%%%%%%%%%%%%%%%%%
\section{Harmonic pencils in affine \\ complete quadrilaterals}\label{sec:affine-quadrilaterals}
This extremely rich configuration has first been described by Jakob Steiner in~\cite{steiner}.
We present the theorem translated from the French original, with a few
additional remarks in parentheses.
\begin{theorem}[Jakob Steiner, 1827]\label{thm:steiner}
%Let $a_1,a_2,a_3,a_4$ be the sides of  a complete quadrilateral $A_1,A_2,\ldots,A_6$.
%Then there holds:
For a complete quadrilateral with four sides and six vertices there holds:
\begin{enumerate}
% 1
\item The four sides, taken in groups of three, form four triangles. Their circumcircles
pass through a single point $M$ (the Miquel point, see~\cite{moss}, \cite{miquel}).\label{steiner1}

%2
\item The centers of these four circumcircles and $M$ lie on a further circle.

%3
\item The feet of the perpendiculars from $M$ to the four sides %$a_{i}$
of the quadrilateral 
lie on a line $s$ (the Simson-Wallace line), and $M$ is the only point with this property.

%4
\item The orthocenters of the four triangles in~(\ref{steiner1}) lie on a line $c$.
($c$ is the common radical axis of the three Thales circles over the diagonals of the 
complete quadrilateral: Bodenmiller-Steiner Theorem, see~\cite[p.~138]{gudermann}).

%5
\item The lines $s$ and $c$ are parallel. $c$ is the middle line between $M$ and $s$.

%6
\item The midpoints of the three diagonals of the complete quadrilateral
lie on a line $n$ (Gau\ss-Newton Theorem, see~\cite{newton}, \cite{gauss}, \cite[p.~112--121]{werke}).

%7
\item The line $n$ is perpendicular to $c$ and $s$.

%8
\item Each of the four triangles of the quadrilateral has an incircle and three excircles
which gives a total of sixteen circles. The centers $C_k$ of these circles lie in groups of four on
eight  new circles.\label{8}

%9
\item The eight circles mentioned in~(\ref{8}) can be decomposed into two
groups of four circles. The four circles in on group are orthogonal
to the four circles of the other group. In particular, the four centers of
the circles of both groups lie each on a line, and these two lines are orthogonal.
(And the four  circles in one of the two groups meet in two points on 
the line which carries the centers of the other group.)\label{9}

%10
\item The two perpendicular lines mentioned in~(\ref{9}) meet in $M$.
\end{enumerate}
\end{theorem}
We add the following observations to complete the picture and to focus on
the part of the constellation which is relevant for our purposes:
\begin{theorem}\label{steiner-add}
%Let $A_1,A_2,\ldots,A_6$ be the vertices of a complete quadrilateral
%with its four sides $a_{12}=A_1\times A_2$, $a_{23}=A_2\times A_3$, $a_{34}=A_3\times A_4$,
%$a_{14}=A_1\times A_4$, and such that
%$A_5=a_{12}\times a_{34}$, $A_6=a_{14}\times a_{23}$.
Let $A_1,A_2,A_3,A_4$ be the vertices of a quadrilateral with sides
$a_{12}=A_1\times A_2$, $a_{23}=A_2\times A_3$, $a_{34}=A_3\times A_4$, $a_{41}=A_4\times A_1$,
and  let $A_5=a_{12}\times a_{34}$, $A_6=a_{23}\times a_{41}$.
Denote by $g_i$ and $h_i$ the inner and outer angle bisectors in the points
$A_i$, $i\in\{1,\ldots,6\}$ (see Figure~\ref{fig:steiner-add}).
Then there holds:
\begin{enumerate}\setcounter{enumi}{10}
\item The following quintuples of points are collinear:
\begin{alignat*}{6}
&A_5,   \quad& &h_1\times h_4, \quad  & &g_1\times g_4,\quad   & &   h_3\times h_2,  \quad & &  g_3\times g_2 && \text{ (on $g_5$)} \\
&A_6,            & &h_3\times h_4,            & &g_3\times g_4,            & &   h_1\times h_2,            & &  g_1\times g_2 && \text{ (on $g_6$)} \\
&A_5,            & &g_1\times h_4,            & &h_1\times g_4,            & &   g_3\times h_2,             & & h_3\times g_2 && \text{ (on $h_5$)} \\
&A_6,            & &g_3\times h_4,            & &h_3\times g_4,            & &   g_1\times h_2,             & & h_1\times g_2 && \text{ (on $h_6$)} 
\end{alignat*}
Stated differently:
The sixteen centers $C_k$ of the circles mentioned in~(\ref{8}) lie in groups of four
on the angle bisectors of the complete quadrilateral in the points $A_1,A_2,\ldots,A_6$:
In each center $C_k$, three of these angle bisectors meet.\label{11}

\item If $A_i,A_j$ lie on a side of the complete quadrilateral,
then there are four of the centers $C_k$ mentioned in~(\ref{8}), say $C_{k_1},C_{k_2},C_{k_3},C_{k_4}$, such that
$A_i,A_j, C_{k_1},C_{k_2}$ and $A_i,A_j, C_{k_3},C_{k_4}$  each lie on a circle.
This gives rise to 24 new circles: In each center $C_k$, three of theses circles meet,
together with two of the eight circles mentioned in~(\ref{8}).\label{12}

\item Each angle bisector carries two centers of the 24 new circles mentioned in~(\ref{12}).\label{13}
\end{enumerate}
\end{theorem}
\begin{figure}
\begin{center}
\begin{tikzpicture}[line cap=round,line join=round,x=15,y=15]
\clip(-6.722483737131825,-14.450974125619656) rectangle (15.122632785534396,8.871589891333219);
\draw [line width=1pt,domain=-6.722483737131825:15.122632785534396] plot(\x,{(-5.6769440361300685--2.1891203990940795*\x)/0.9645676737448516});
\draw [line width=1pt,domain=-6.722483737131825:15.122632785534396] plot(\x,{(-2.836308217443798--1.5550951107303632*\x)/2.8217176732686307});
\draw [line width=1pt,domain=-6.722483737131825:15.122632785534396] plot(\x,{(--25.58755234505827-3.7021456726354036*\x)/2.187631533830012});
\draw [line width=1pt,domain=-6.722483737131825:15.122632785534396] plot(\x,{(--9.811128985699225-0.042069837189038894*\x)/-5.973916880843494});
\draw [line width=0.6pt,dash pattern=on 3pt off 3pt,color=blue,domain=-6.722483737131825:15.122632785534396] plot(\x,{(--1.5500482833975386--0.9675696301361109*\x)/-0.25260445530170184});
\draw [line width=0.6pt,dash pattern=on 3pt off 3pt,color=red,domain=-6.722483737131825:15.122632785534396] plot(\x,{(--1.3012248432222944-0.25260445530170184*\x)/-0.9675696301361109});
\draw [line width=0.6pt,color=darkgreen,domain=-6.722483737131825:15.122632785534396] plot(\x,{(-0.6725689431518097--0.8356909587597577*\x)/-0.5491999831092468});
\draw [line width=0.6pt,color=red,domain=-6.722483737131825:15.122632785534396] plot(\x,{(--2.391428590155873-0.5491999831092468*\x)/-0.8356909587597577});
\draw [line width=0.6pt,color=darkgreen,domain=-6.722483737131825:15.122632785534396] plot(\x,{(--7.61285204681938-0.8702813942472466*\x)/-0.49255486478875493});
\draw [line width=0.6pt,color=red,domain=-6.722483737131825:15.122632785534396] plot(\x,{(--2.485054561626597-0.49255486478875493*\x)/0.8702813942472466});
\draw [line width=0.6pt,color=darkgreen,domain=-6.722483737131825:15.122632785534396] plot(\x,{(--3.5324318008431437-0.2635438288229937*\x)/0.9646474227868527});
\draw [line width=0.6pt,color=red,domain=-6.722483737131825:15.122632785534396] plot(\x,{(-4.904141098258687--0.9646474227868527*\x)/0.2635438288229937});
\draw [line width=0.6pt,dash pattern=on 3pt off 3pt,color=blue,domain=-6.722483737131825:15.122632785534396] plot(\x,{(--3.915746593748496--0.05930666686692026*\x)/0.9982398104990283});
\draw [line width=0.6pt,dash pattern=on 3pt off 3pt,color=red,domain=-6.722483737131825:15.122632785534396] plot(\x,{(-4.678283949651871--0.9982398104990283*\x)/-0.05930666686692026});
\draw [line width=0.6pt,color=darkgreen,domain=-6.722483737131825:15.122632785534396] plot(\x,{(-1.7171784701005388--0.7377528746808143*\x)/0.67507088213031});
\draw [line width=0.6pt,color=red,domain=-6.722483737131825:15.122632785534396] plot(\x,{(-2.3303429150713364--0.67507088213031*\x)/-0.7377528746808143});
\begin{small}
\draw [fill=white] (1.8754323262551482,-1.6291203990940792) circle (1.8pt);
\draw[color=black] (1.99,-2.3) node {$A_1$};
\draw [fill=white] (7.849349207098642,-1.5870505619050403) circle (1.8pt);
\draw[color=black] (8.456530400598407,-1.2) node {$A_2$};
\draw [fill=white] (5.6617176732686305,2.1150951107303633) circle (1.8pt);
\draw[color=black] (5.43,1.6) node {$A_3$};
\draw [fill=white] (2.84,0.56) circle (1.8pt);
\draw[color=black] (2.2,0.75) node {$A_4$};
\draw[color=black] (5.82,6) node {$a_{41}$};
\draw[color=black] (-5.8,-4.66) node {$a_{34}$};
\draw[color=black] (2.8,6) node {$a_{23}$};
\draw[color=black] (-5.8,-1.4) node {$a_{12}$};
\draw [fill=white] (4.437819404590411,4.186307565416579) circle (1.8pt);
\draw[color=black] (3.78,3.82) node {$A_6$};
\draw [fill=white] (-1.1710844515128602,-1.6505747425994877) circle (1.8pt);
\draw[color=black] (-1.7,-1.3) node {$A_5$};
\draw[color=blue] (-3.3,8.38) node {$h_5$};
\draw[color=red] (-5.8,-2.55) node {$g_5$};
\draw[color=darkgreen] (-4.18,8.38) node {$h_1$};
\draw[color=red] (-5.8,-6.3) node {$g_1$};
\draw[color=darkgreen] (13.95,8.38) node {$h_2$};
\draw[color=red] (-5.8,6.58) node {$g_2$};
\draw[color=darkgreen] (-5.8,5.6) node {$h_3$};
\draw[color=red] (7.8,8.38) node {$g_3$};
\draw[color=blue] (-5.8,3.9) node {$h_6$};
\draw[color=red] (4.65,8.38) node {$g_6$};
\draw[color=darkgreen] (10.6,8.38) node {$h_4$};
\draw[color=red] (-5.8,8.03) node {$g_4$};
\draw [fill=blue] (3.829192211301657,-0.3451470778324675) circle (2.0pt);
%\draw[color=blue] (3.6900033654594817,-0.71423016088346) node {$G$};
\draw [fill=blue] (4.674052108243779,0.2100785546416566) circle (2.0pt);
%\draw[color=blue] (4.323195148725168,0.868749297280762) node {$H$};
\draw [fill=blue] (5.078695061870641,-0.018938010148624364) circle (1.8pt);
%\draw[color=blue] (5.65993335784162,0.23555751401507324) node {$I$};
\draw [fill=blue] (4.757505051300725,-1.1945873027346006) circle (1.8pt);
%\draw[color=blue] (5.237805502331162,-0.7494074821759984) node {$J$};
\draw [fill=blue] (-2.7544789098391234,4.41441880065568) circle (1.8pt);
%\draw[color=blue] (-3.3806382143406966,4.245772141364436) node {$K$};
\draw [fill=blue] (-0.7839796304612746,3.8760739997335985) circle (1.8pt);
%\draw[color=blue] (-0.5312751896451025,4.456836069119666) node {$L$};
\draw [fill=blue] (-1.7064580958854416,3.8212684084664352) circle (1.8pt);
%\draw[color=blue] (-2.008722683931707,3.4015164303435172) node {$M$};
\draw [fill=blue] (-3.188299709334126,6.076114870154515) circle (1.8pt);
%\draw[color=blue] (-2.9585103588302384,6.673007310549576) node {$N$};
\draw [fill=blue] (11.34930746709747,4.596927654664364) circle (1.8pt);
%\draw[color=blue] (11.604900656280577,4.245772141364436) node {$O$};
\draw [fill=blue] (6.257106910434643,4.294393695532263) circle (1.8pt);
%\draw[color=blue] (6.64489835403269,3.9643535710241298) node {$P$};
\draw [fill=blue] (9.371097691366563,1.1016842111243614) circle (1.8pt);
%\draw[color=blue] (9.529438700020822,0.6576853695255325) node {$Q$};
\draw [fill=blue] (4.5427097739007625,2.4208105662598034) circle (1.8pt);
%\draw[color=blue] (4.745323004235627,2.0999555425196017) node {$R$};
\draw [fill=blue] (-0.7297176717322211,-3.3411747415451205) circle (1.8pt);
%\draw[color=blue] (-0.35538858318241145,-3.4932385429939834) node {$S$};
\draw [fill=blue] (5.072331091108399,-6.493686087412137) circle (1.8pt);
%\draw[color=blue] (5.589578715256544,-6.237069603811968) node {$T$};
\draw [fill=blue] (1.441319287727939,-0.9685516581355607) circle (1.8pt);
%\draw[color=blue] (1.6848960517848042,-0.4328115905431538) node {$U$};
\draw [fill=blue] (1.6650293055477274,-12.51395217458321) circle (1.8pt);
%\draw[color=blue] (2.1070239072952623,-12.393100830006166) node {$V$};
\end{small}
\end{tikzpicture}
\caption{Complete quadrilateral with angle bisectors.}\label{fig:steiner-add}
\end{center}
\end{figure}
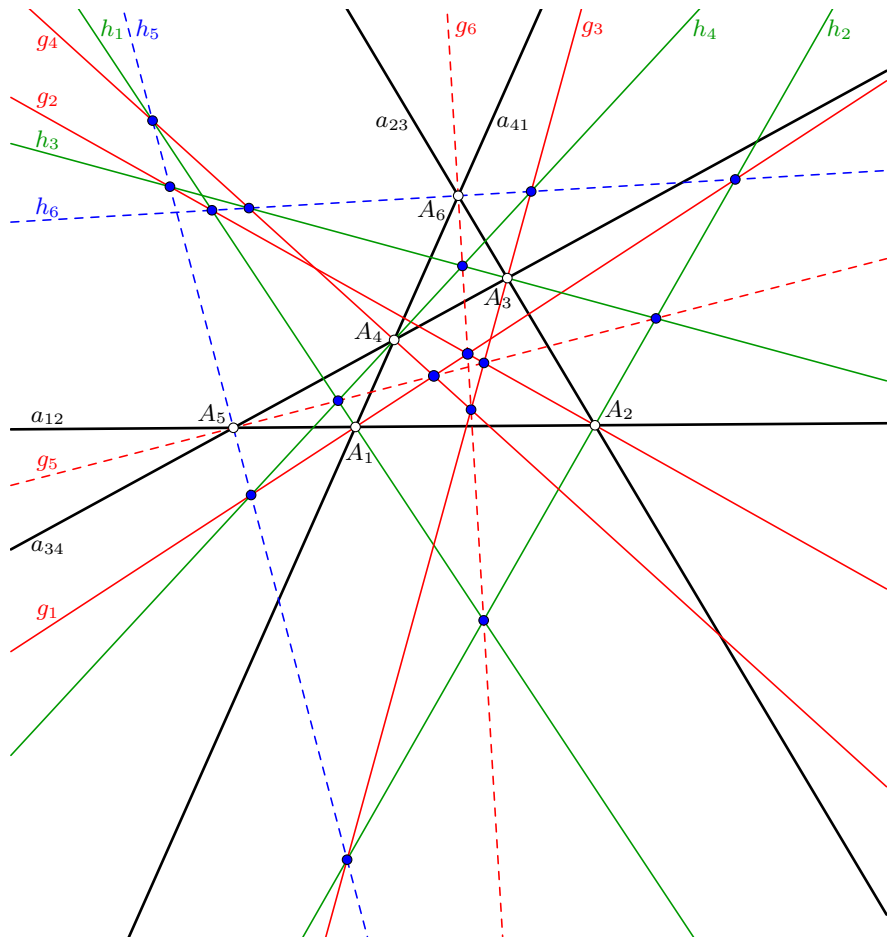
\begin{proof}
\begin{enumerate}\setcounter{enumi}{10}
\item It is clear from the construction, that the centers $C_k$ of the incircles and excircles lie on the angle
bisectors in the points $A_1,A_2,\ldots, A_6$: See Figure~\ref{fig:steiner-add}.
%Since we want to analyze the situation later in more detail, we code
%the centers $C_k$ differently by $H_{ijk}$ and $G_{ijk}$
%
\item[(\ref{12}),(\ref{13})] Since the two angle
bisectors in each of the points $A_1,\ldots,A_6$ are orthogonal, the 24 new circles 
are just Thales circles over two centers. See Figure~\ref{fig:xii}, where the situation
is illustrated for the segment $A_2,A_5$ and two centers $C$ and $D$.
\end{enumerate}
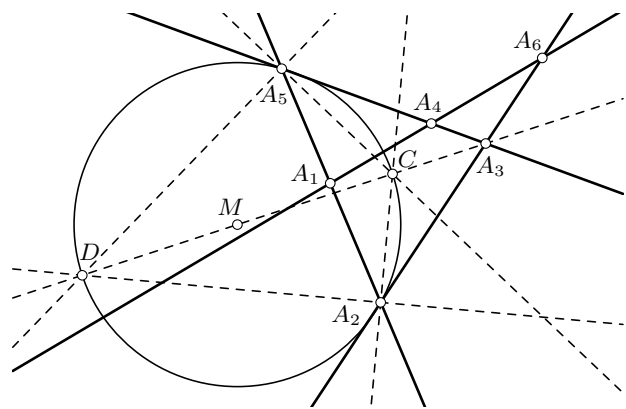
\begin{figure}
\begin{center}
\begin{tikzpicture}[line cap=round,line join=round,x=15,y=15]
\clip(-7.5194220039180495,-4.823896912084252) rectangle (7.705043944318036,5.048434376753706);
\draw [line width=1.pt,domain=-7.5194220039180495:7.705043944318036] plot(\x,{(--4.284869366273985-5.879028745038336*\x)/2.4680828222094906});
\draw [line width=1.pt,domain=-7.5194220039180495:7.705043944318036] plot(\x,{(--17.11185815961716-1.8818709934899476*\x)/5.088455228921451});
\draw [line width=1.pt,domain=-7.5194220039180495:7.705043944318036] plot(\x,{(-12.443572292679814--3.9971577515483885*\x)/2.6203724067119603});
\draw [line width=1.pt,domain=-7.5194220039180495:7.705043944318036] plot(\x,{(-2.8578021817808796-3.1525571669214374*\x)/-5.289336982406399});
\draw [line width=.6pt,dash pattern=on 3pt off 3pt,domain=-7.5194220039180495:7.705043944318036] plot(\x,{(--2.0855271630598624-0.6916565513673865*\x)/0.7222265676022822});
\draw [line width=.6pt,dash pattern=on 3pt off 3pt,domain=-7.5194220039180495:7.705043944318036] plot(\x,{(--3.116644907953579--0.7222265676022822*\x)/0.6916565513673865});
\draw [line width=.6pt,dash pattern=on 3pt off 3pt,domain=-7.5194220039180495:7.705043944318036] plot(\x,{(-2.0564191715460303-0.09127556098663005*\x)/0.9958256734823502});
\draw [line width=.6pt,dash pattern=on 3pt off 3pt,domain=-7.5194220039180495:7.705043944318036] plot(\x,{(-1.855072083384389--0.9958256734823502*\x)/0.09127556098663005});
\draw [line width=.6pt,dash pattern=on 3pt off 3pt,domain=-7.5194220039180495:7.705043944318036] plot(\x,{(-0.8609973878830894-0.7654211864405276*\x)/-2.3241573403643527});
\draw [line width=0.6pt] (-1.9146527974449943,-0.26010219594864376) circle (4.074990399900657);
\begin{small}
\draw [color=black,fill=white] (-0.8084552289214507,3.6618709934899476) circle (1.8pt);
\draw[color=black] (-1.0,3.) node {$A_5$};
\draw [color=black,fill=white] (1.65962759328804,-2.2171577515483882) circle (1.8pt);
\draw[color=black] (0.8,-2.55) node {$A_2$};
\draw [color=black,fill=white] (4.28,1.78) circle (1.8pt);
\draw[color=black] (4.4,1.2) node {$A_3$};
\draw [color=black,fill=white] (5.690881625445003,3.9321812743941766) circle (1.8pt);
\draw[color=black] (5.334020558937006,4.36) node {$A_6$};
\draw [color=black,fill=white] (0.40154464303860393,0.7796241074727392) circle (1.8pt);
\draw[color=black] (-0.22,1.) node {$A_1$};
\draw [color=black,fill=white] (2.9223749455130172,2.2820925005988104) circle (1.8pt);
\draw[color=black] (2.838206469062238,2.8) node {$A_4$};
\draw [color=black,fill=white] (1.9558426596356475,1.0145788135594724) circle (1.8pt);
\draw[color=black] (2.325178017254647,1.4) node {$C$};
\draw [color=black,fill=white] (-1.9146527974449943,-0.26010219594864376) circle (1.8pt);
\draw[color=black] (-2.0840935415241097,0.18) node {$M$};
\draw [color=black,fill=white] (-5.785148254525636,-1.5347832054567592) circle (1.8pt);
\draw[color=black] (-5.6,-0.95) node {$D$};
\end{small}
\end{tikzpicture}
\caption{The center $C$ of the incircle of the triangle $A_2A_3A_5$, the 
center $D$ of the excircle opposite of its vertex $A_3$, and the vertices $A_2,A_5$ lie
on a circle whose center $M$ is the midpoint of the segment $CD$.}\label{fig:xii}
\end{center}
\end{figure}
\end{proof}
Tow sides of the quadrilateral meeting in $A_i$ form together with the angle bisectors
in this point a harmonic pencil of lines. According to Theorem~\ref{steiner-add}\,(\ref{11}),
groups of four of the intersection points of angle bisectors and one of the points $A_i$ are collinear.
In Section~\ref{sec:projective-quadrilaterals} we will consider this situation
where we replace the angle bisectors by harmonic lines.
%%%%%%%%%%%%%%%%%%%%%%%%%%%%%%%%%%%%%%%%%
\section{Harmonic pencils in projective triangles}\label{sec:projective-triangles}
We now replace the angle bisectors in a triangle by lines such
that in each vertex a harmonic pencil of lines results. Each pair of 
these three harmonic pencils shares a common line (a side
of the triangle), and they are therefore not completely independent.
And indeed, certain intersection points of  lines turn out to be collinear.
\begin{proposition}\label{prop:free-triangle}
Let $A_1,A_2,A_3$ be a triangle with sides $a_i$, $i\in\{1,2,3\}$ and such that
$A_i\notin a_1$. Through each vertex $A_i$ we choose two lines $g_i, h_i$
such that  $(a_ia_{i+1}; g_{i+2}h_{i+2})$ is a harmonic pencil of lines (indices are taken cyclically). 
Then, there holds:
\begin{enumerate}
\item The following triples of points are collinear:
\begin{alignat*}{4}
&\text{on line $u_i$:\quad}&& A_i,\quad && g_{i+1}\times g_{i+2},\quad && h_{i+1}\times h_{i+2}\\
&\text{on line $v_i$:\quad}&& A_i,\quad && g_{i+1}\times h_{i+2},\quad && g_{i+2}\times h_{i+1}
\end{alignat*}
for $i\in\{1,2,3\}$ %, indices are taken cyclically 
(see the dashed lines in Figure~\ref{fig:free-triangle}).\label{prop:5i}
\item The lines $(a_ia_{i+1}; u_{i+2}v_{i+2})$ are  harmonic pencils.\label{prop:5ii}
\end{enumerate}
\end{proposition}
\begin{proof}
(\ref{prop:5i}): This follows directly by applying Corollary~\ref{cor:two-pencils} to the three pairs of points $A_i,A_j$, $i\neq j$.

(\ref{prop:5ii}): The harmonic pencil $(a_1  a_3;h_2g_2)$ cuts the line $h_3$ 
in harmonic points, i.e., the points 
$$
( a_1\times h_3,a_3\times h_3;  h_2\times h_3,g_2\times h_3)=
( a_2\times h_3,a_3\times h_3;  u_1\times h_3,v_1\times h_3)
$$
are harmonic.
But this implies that $(a_2a_3;u_1v_1)$ is a harmonic pencil. The remaining cases are analogous.
\end{proof}
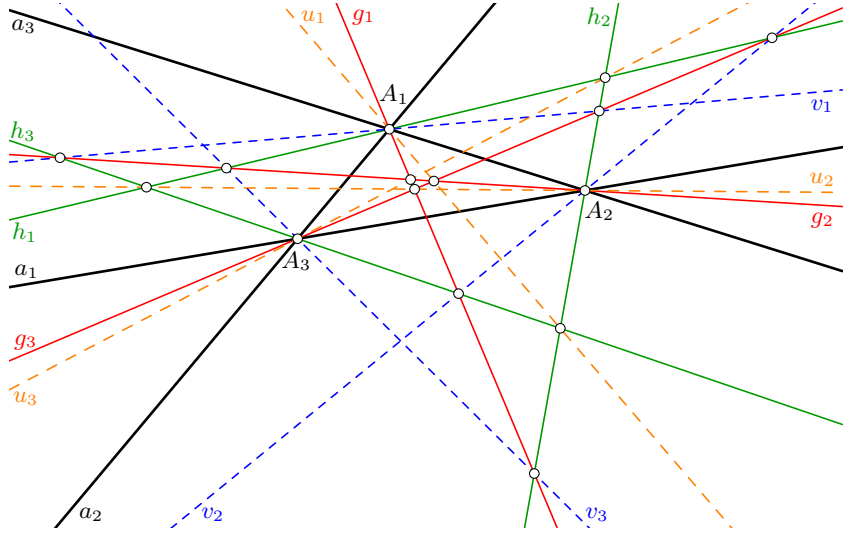
\begin{figure}[h]
\begin{center}
\begin{tikzpicture}[line cap=round,line join=round,x=12,y=12]
\clip(-9.350402562316095,-10.451233605117318) rectangle (16.62751613874593,6.011177796454527);
\draw [line width=1.pt,domain=-9.350402562316095:16.62751613874593] plot(\x,{(--2.6881058556047908-3.444263378343046*\x)/-2.860489924386605});
\draw [line width=1.pt,domain=-9.350402562316095:16.62751613874593] plot(\x,{(-11.76147824023111--1.517810980286766*\x)/8.960922518231506});
\draw [line width=1.pt,domain=-9.350402562316095:16.62751613874593] plot(\x,{(-17.448721864854353--1.9264523980562798*\x)/-6.100432593844902});
\draw [line width=.6pt,color=red,domain=-9.350402562316095:16.62751613874593] plot(\x,{(-7.138942756006888--2.499264643951288*\x)/5.85298591662718});
\draw [line width=.6pt,color=red,domain=-9.350402562316095:16.62751613874593] plot(\x,{(-5.013777787476932--0.4602103886426474*\x)/-7.3181590084426675});
\draw [line width=.6pt,color=red,domain=-9.350402562316095:16.62751613874593] plot(\x,{(--9.30094554960268-2.7642754501877045*\x)/1.1540541899151329});
\draw [line width=.6pt,color=darkgreen,domain=-9.350402562316095:16.62751613874593] plot(\x,{(--232.51535198083815--53.38250532109403*\x)/-155.27727622765858});
\draw [line width=.6pt,color=darkgreen,domain=-9.350402562316095:16.62751613874593] plot(\x,{(-101.13561095310841--11.794227568961304*\x)/2.094838310805054});
\draw [line width=.6pt,color=darkgreen,domain=-9.350402562316095:16.62751613874593] plot(\x,{(-60.90382758467995-9.983322098713925*\x)/-41.46608660042247});
\draw [line width=.6pt,dash pattern=on 3pt off 3pt,color=blue,domain=-9.350402562316095:16.62751613874593] plot(\x,{(--31.14530026355007--1.4685530218266598*\x)/16.81405177178579});
\draw [line width=.6pt,dash pattern=on 4pt off 4pt,color=orange,domain=-9.350402562316095:16.62751613874593] plot(\x,{(-19.74039867713516--4.630181005672508*\x)/-3.942627757868677});
\draw [line width=.6pt,dash pattern=on 4pt off 4pt,color=orange,domain=-9.350402562316095:16.62751613874593] plot(\x,{(-1.109963173450977--0.039946143236323534*\x)/-5.312837259404464});
\draw [line width=.6pt,dash pattern=on 3pt off 3pt,color=blue,domain=-9.350402562316095:16.62751613874593] plot(\x,{(--67.76491479301258-8.043038713238316*\x)/-9.776427468424014});
\draw [line width=.6pt,dash pattern=on 4pt off 4pt,color=orange,domain=-9.350402562316095:16.62751613874593] plot(\x,{(-4.167533516023573--1.8598456341652405*\x)/3.5219666727041856});
\draw [line width=.6pt,dash pattern=on 3pt off 3pt,color=blue,domain=-9.350402562316095:16.62751613874593] plot(\x,{(-3.850194802261763-2.220967689938769*\x)/2.220512616824573});
\begin{small}
\draw [color=black,fill=white] (2.500198552999839,2.070706982248501) circle (1.8pt);
\draw[color=black] (2.65,3.16) node {$A_1$};
\draw [color=black,fill=white] (-0.3602913713867657,-1.3735563960945447) circle (1.8pt);
\draw[color=black] (-0.4,-2.074084540842284) node {$A_3$};
\draw[color=black] (-6.796393701256626,-10.042592187347804) node {$a_2$};
\draw[color=blue] (-3.,-10.042592187347804) node {$v_2$};
\draw[color=blue] (8.98,-10.042592187347804) node {$v_3$};
\draw [color=black,fill=white] (8.60063114684474,0.14425458419222129) circle (1.8pt);
\draw[color=black] (9.05,-0.5) node {$A_2$};
\draw[color=orange] (0.15,5.6) node {$u_1$};

\draw[color=black] (-8.810412117406377,-2.4) node {$a_1$};
\draw[color=black] (-8.9,5.3) node {$a_3$};
\draw[color=red] (-8.8396007901042,-4.613499065552835) node {$g_3$};
\draw[color=orange] (-8.8396007901042,-6.4) node {$u_3$};
\draw[color=red] (16,-.8) node {$g_2$};
\draw[color=orange] (16,.5) node {$u_2$};
\draw[color=blue] (16,2.8) node {$v_1$};

\draw[color=red] (1.7,5.6) node {$g_1$};
\draw[color=darkgreen] (-8.927166808197669,2.0123296368528556) node {$h_3$};
\draw[color=darkgreen] (9.05305557366099,5.6) node {$h_2$};
\draw[color=darkgreen] (-8.868789462802022,-1.1692356872097887) node {$h_1$};
\draw [color=black,fill=white] (-7.769733197672582,1.1737227507790087) circle (1.8pt);
\draw [color=black,fill=white] (9.044318574113207,2.6422757726056685) circle (1.8pt);
\draw [color=black,fill=white] (7.830874169240506,-4.189582982092903) circle (1.8pt);
\draw [color=black,fill=white] (3.8882464113718287,0.44059802357960526) circle (1.8pt);
\draw [color=black,fill=white] (3.287793887440276,0.18420072742854482) circle (1.8pt);
\draw [color=black,fill=white] (14.434899505356684,4.944090354722148) circle (1.8pt);
\draw [color=black,fill=white] (4.658472036932669,-3.0989483585161675) circle (1.8pt);
\draw [color=black,fill=white] (3.16167530131742,0.4862892380706958) circle (1.8pt);
\draw [color=black,fill=white] (9.230608172898167,3.6911118142550787) circle (1.8pt);
\draw [color=black,fill=white] (-5.074327245512705,0.24707381618072521) circle (1.8pt);
\draw [color=black,fill=white] (-2.5808039882113385,0.8474112938442243) circle (1.8pt);
\draw [color=black,fill=white] (7.0199329897594716,-8.755293264604184) circle (1.8pt);
\end{small}
\end{tikzpicture}
\caption{Harmonic pencils in a triangle.}\label{fig:free-triangle}
\end{center}
\end{figure}
If we compare the situation in Section~\ref{sec:affine-triangles} of a triangle with angle bisectors 
to the Proposition~\ref{prop:free-triangle} above, we observe that
here neither the lines $g_1,g_2,g_3$ nor the lines $g_i,h_{i+1}, h_{i+2}$ are concurrent in general.
However, the next theorem shows that one such concurrency already implies all other possible concurrencies.
\begin{theorem}\label{thm:projective-triangle}
Let $A_1,A_2,A_3$ be a triangle with sides $a_i$, $i\in\{1,2,3\}$ and such that
$A_i\notin a_1$. Through each vertex $A_i$ we choose two lines $g_i, h_i$
such that  $(a_ia_{i+1}; g_{i+2}h_{i+2})$ is a harmonic pencil of lines (indices are taken cyclically). Then we have:
\begin{enumerate}
\item If the lines $g_1,g_2,g_3$ are concurrent, then so are the triples of lines $g_i,h_{i+1},h_{i+2}$
for $i\in\{1,2,3\}$. \label{pt1}
\item Vice versa, if one of the triples  $g_i,h_{i+1},h_{i+2}$ is concurrent,
then so are the other two triples and the triple $g_1,g_2,g_3$. \label{pt2}
\item Each of the concurrencies in (\ref{pt1}) and (\ref{pt2}) occurs if and only if
$$
\frac{A_1 B_3}{B_3A_2}\cdot
\frac{A_2 B_1}{B_1A_3}\cdot
\frac{A_3 B_2}{B_2A_1}
=1
$$
where $B_i=a_i\times g_i$ (see Figure~\ref{fig:tritra}).\label{pt3}
\end{enumerate}
\end{theorem}
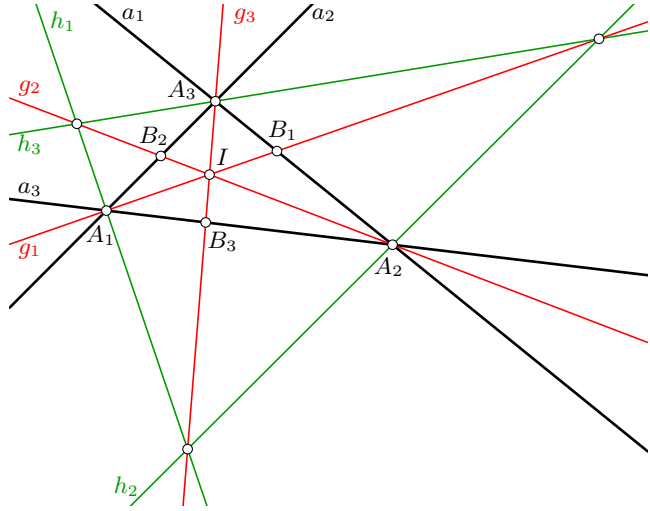
\begin{figure}[h]
\begin{center}
\begin{tikzpicture}[line cap=round,line join=round,x=15,y=15]
\clip(-3.9123977116524986,-8.505459862669417) rectangle (12.005976384406718,4.113136109830486);
\draw [line width=1.pt,domain=-3.9123977116524986:12.005976384406718] plot(\x,{(-9.0040312-0.85904091*\x)/7.138671121271159});
\draw [line width=1.pt,domain=-3.9123977116524986:12.005976384406718] plot(\x,{(-11.780293475509337--3.6088560*\x)/-4.419409519449142});
\draw [line width=1.pt,domain=-3.912397711:12.005976384] plot(\x,{(-1.1816579515-2.74981510*\x)/-2.719261601822016});
\draw [line width=.6pt,color=red,domain=-3.9123977116524986:12.005976384406718] plot(\x,{(-2.35825106369892--1.4956022806466225*\x)/4.25517216690381});
\draw [line width=.6pt,color=red,domain=-3.9123977116524986:12.005976384406718] plot(\x,{(-1.3781438798811947--2.232627129929268*\x)/-5.780346975731895});
\draw [line width=.6pt,color=red,domain=-3.9123977116524986:12.005976384406718] plot(\x,{(--2.008579121725494-1.8454272455083909*\x)/-0.1461670773365933});
\draw [line width=.6pt,color=darkgreen,domain=-3.9123977116524986:12.005976384406718] plot(\x,{(--12.790772221602968--6.862999908092251*\x)/-2.317749570368075});
\draw [line width=.6pt,color=darkgreen,domain=-3.9123977116524986:12.005976384406718] plot(\x,{(-15.894442082577685-1.7747394271968706*\x)/-10.822187135841213});
\draw [line width=.6pt,color=darkgreen,domain=-3.9123977116:12.005976] plot(\x,{(-7.6113 + 1.00555*\x});
\begin{small}
\draw [fill=white] (-1.4986711212711588,-1.0809590846609975) circle (1.8pt);
\draw[color=black] (-1.64,-1.7) node {$A_1$};
\draw [fill=white] (5.64,-1.94) circle (1.8pt);
\draw[color=black] (5.543910892436422,-2.5) node {$A_2$};
\draw[color=black] (-3.4082649427753835,-0.55) node {$a_3$};
\draw [fill=white] (1.2205904805508574,1.668856018305084) circle (1.8pt);
\draw[color=black] (0.39,1.95) node {$A_3$};
\draw[color=black] (-0.8112173455296376,3.8381545995338775) node {$a_1$};
\draw[color=black] (3.924575331800839,3.8381545995338775) node {$a_2$};
\draw [fill=white] (2.756501045632651,0.41464319598562493) circle (1.8pt);
\draw[color=black] (2.9,0.9355719908474592) node {$B_1$};
\draw[color=red] (-3.4082649427753835,-2.1) node {$g_1$};
\draw [fill=white] (-0.14034697573189514,0.292627129929268) circle (1.8pt);
\draw[color=black] (-0.32,.8) node {$B_2$};
\draw[color=red] (-3.4082649427753835,2) node {$g_2$};
\draw [fill=white] (1.074423403214264,-0.17657122720330698) circle (1.8pt);
\draw[color=black] (1.4,0.2786717162500067) node {$I$};
\draw[color=red] (1.9691512585805135,3.80760109838981) node {$g_3$};
\draw [fill=white] (0.9791744323209726,-1.3791337275695272) circle (1.8pt);
\draw[color=black] (1.4,-1.85) node {$B_3$};
\draw[color=darkgreen] (-2.5833204118855586,3.6853870938135396) node {$h_1$};
\draw[color=darkgreen] (-3.4082649427753835,0.58) node {$h_3$};
\draw[color=black] (-3.820737208220296,4.37284086955506) node {$J$};
\draw[color=darkgreen] (-1,-8.016603844364337) node {$h_2$};
\draw [fill=white] (-2.2359006513716055,1.10202317623907) circle (1.8pt);
\draw [fill=white] (0.5275741714859218,-7.080796787470197) circle (1.8pt);
\draw [fill=white] (10.789491340929208,3.2380679916971404) circle (1.8pt);
\end{small}
\end{tikzpicture}
\caption{Harmonic pencils in the vertices of a triangle.}\label{fig:tritra}
\end{center}
\end{figure}
\begin{proof}
(\ref{pt1}) Let $I$ be the intersection point of $g_1,g_2,g_3$, $I_{12}=g_1\times h_2$ and 
$I_{13}=g_1\times h_3$. We want to show that $I_{12}=I_{13}$. Since $(a_ia_{i+1}; g_{i+2}h_{i+2})$ are harmonic pencils
for $i\in\{2,3\}$, we have
\begin{eqnarray*}
\frac{A_1 I}{IB_1} &=&-\frac{A_1 I_{12}}{I_{12}B_1}\\
\frac{A_1 I}{IB_1} &=&-\frac{A_1 I_{13}}{I_{13}B_1}.
\end{eqnarray*}
Hence $\frac{A_1 I_{12}}{I_{12}B_1}=\frac{A_1 I_{13}}{I_{13}B_1}$ and therefore $I_{12}=I_{13}$.
Thus, $g_1,h_2,h_3$ are concurrent. The remaining cases are analogous.

(\ref{pt2}) This follows from (\ref{pt1}) by exchanging the role of $g_i$ and $h_i$
in two of the vertices $A_i$.

(\ref{pt3}) This is Ceva's Theorem.
\end{proof}
%%%%%%%%%%%%%%%%%%%%%%%%%%%%%%%%%%%%%%%%%
\section{Harmonic pencils in projective \\ complete quadrilaterals}\label{sec:projective-quadrilaterals}
We consider a quadrilateral like in Section~\ref{sec:affine-quadrilaterals} and
replace the angle bisectors by lines which form together with the sides of the quadrilateral
harmonic pencils. Since neighboring pencils share a common side they are not
completely independent, and certain intersection points of the lines turn out to
be collinear:
\begin{proposition}\label{prop:free-quadrilateral}
Let  $A_1, A_2, A_3, A_4$ be the vertices of a quadrilateral with sides
$a_{i(i+1)}=A_i\times A_{i+1}$ for $i\in\{1,2,3,4\}$ (where indices are taken cyclically).
The quadrilateral is completed with $A_5=a_{12}\times a_{34}$, $A_6=a_{41}\times a_{23}$.
%
%In this section, we consider a complete quadrilateral with four sides $a_{12},a_{23},a_{34},a_{14}$ and vertices
%\begin{alignat*}{3}
%A_1&=a_{12}\times a_{14} \qquad& A_2&=a_{12}\times a_{23} \qquad& A_3&=a_{23}\times a_{34} \\
%A_4&=a_{14}\times a_{34} & A_5&=a_{12}\times a_{34} & A_6&=a_{14}\times a_{23} 
%\end{alignat*}
%
We assume that in each vertex $A_1,A_2,A_3,A_4$ two lines $g_i, h_i$ 
build together with the sides a harmonic pencil $(a_{(i-1)i} a_{i(i+1)};g_i h_i)$.
Then the following triples of points are collinear:
\begin{alignat*}{6}
&A_5,\quad && g_1\times h_4,\quad && h_1\times g_4,\qquad\qquad  &&A_6,\quad && g_1\times h_2,\quad && h_1\times g_2,\\
&A_5,\quad && g_3\times h_2,\quad && h_3\times g_2,\qquad\qquad  &&A_6,\quad && g_3\times h_4,\quad && h_3\times g_4,\\
&A_5,\quad && g_1\times g_4,\quad && h_1\times h_4,\qquad\qquad  &&A_6,\quad && g_1\times g_2,\quad && h_1\times h_2,\\
&A_5,\quad && g_2\times g_3,\quad && h_2\times h_3,\qquad\qquad  &&A_6,\quad && g_3\times g_4,\quad && h_3\times h_4\\
\end{alignat*}
(see the dashed lines in Figure~\ref{fig:free-quadrilateral}).
\end{proposition}
\begin{proof}
This follows directly by applying Corollary~\ref{cor:two-pencils} to neighboring pairs of vertices of the 
quarilateral $A_1,A_2,A_3,A_4$.
\end{proof}
\begin{figure}[h]
\begin{center}
\begin{tikzpicture}[line cap=round,line join=round,x=10,y=10]
\clip(-9.558188933527463,-13.044783003455633) rectangle (22.8,12.747433541487382);
\draw [line width=1.pt,domain=-9.558188933527463:24.5570478440037] plot(\x,{(-1.2607910721540152-1.4594336348911212*\x)/4.874508340536357});
\draw [line width=1.pt,domain=-9.558188933527463:24.5570478440037] plot(\x,{(--13.425488873459358-3.473452051040869*\x)/2.0723957615453967});
\draw [line width=1.pt,domain=-9.558188933527463:24.5570478440037] plot(\x,{(--55.58711145330577-0.37945274507169113*\x)/-10.18684677154005});
\draw [line width=1.pt,domain=-9.558188933527463:24.5570478440037] plot(\x,{(-0.9583898851673018--5.312338431003681*\x)/3.239942669458297});
\draw [line width=1.pt,color=red,domain=-9.558188933527463:24.5570478440037] plot(\x,{(-27.89112169863715--4.609093669853599*\x)/7.663279497090336});
\draw [line width=1.pt,color=darkgreen,domain=-9.558188933527463:24.5570478440037] plot(\x,{(--1043.1625099763473--168.17689945318602*\x)/-89.92759622237712});
\draw [line width=1.pt,color=red,domain=-9.558188933527463:24.5570478440037] plot(\x,{(--0.6131138989436591-5.263823700840051*\x)/-1.937508759680848});
\draw [line width=1.pt,color=darkgreen,domain=-9.558188933527463:24.5570478440037] plot(\x,{(--1.3596802288324006-5.368723738748694*\x)/-4.753671315843616});
\draw [line width=1.pt,color=red,domain=-9.558188933527463:24.5570478440037] plot(\x,{(--18.009230600054543--3.327977104164439*\x)/-7.925721973379549});
\draw [line width=1.pt,color=darkgreen,domain=-9.558188933527463:24.5570478440037] plot(\x,{(--322.3268669748869-26.695969513555102*\x)/-26.237030075395356});
\draw [line width=1.pt,color=red,domain=-9.558188933527463:24.5570478440037] plot(\x,{(--17.106261464387085-3.545382012364622*\x)/0.14135295370002687});
\draw [line width=1.pt,color=darkgreen,domain=-9.558188933527463:24.5570478440037] plot(\x,{(--1137.4735532434042-357.415075607428*\x)/354.7920382951311});
\draw [line width=.6pt,dash pattern=on 3pt off 3pt,color=blue,domain=-9.558188933527463:24.5570478440037] plot(\x,{(-26.111635813655866--0.46053790566055586*\x)/3.8922952989896364});
\draw [line width=.6pt,dash pattern=on 3pt off 3pt,color=orange,domain=-9.558188933527463:24.5570478440037] plot(\x,{(-10.150615724279898-0.038446759184029844*\x)/2.200970675938202});
\draw [line width=.6pt,dash pattern=on 3pt off 3pt,color=orange,domain=-9.558188933527463:24.5570478440037] plot(\x,{(--12.054626470030534-9.982093909270795*\x)/-2.7305129652533764});
\draw [line width=.6pt,dash pattern=on 3pt off 3pt,color=blue,domain=-9.558188933527463:24.5570478440037] plot(\x,{(--1.7147457198731635--0.8008195086738374*\x)/1.0971898154761464});
\draw [line width=.6pt,dash pattern=on 3pt off 3pt,color=blue,domain=-9.558188933527463:24.5570478440037] plot(\x,{(-34.49287108795462--0.9006140304383488*\x)/4.2172240268645425});
\draw [line width=.6pt,dash pattern=on 3pt off 3pt,color=orange,domain=-9.558188933527463:24.5570478440037] plot(\x,{(--32.15132414333358--0.38065060634781656*\x)/-7.790236211708326});
\draw [line width=.6pt,dash pattern=on 3pt off 3pt,color=orange,domain=-9.558188933527463:24.5570478440037] plot(\x,{(-15.451060625310042--9.209376695795576*\x)/1.1014751229443607});
\draw [line width=.6pt,dash pattern=on 3pt off 3pt,color=blue,domain=-9.558188933527463:24.5570478440037] plot(\x,{(--45.3655101491843--5.804992511570754*\x)/18.737679436612908});
\begin{small}
\draw [color=black,fill=white] (0.019161373684926764,-0.2643868335772935) circle (1.8pt);
\draw[color=black] (-1.2,-0.48) node {$A_1$};
\draw [color=black,fill=white] (4.893669714221284,-1.7238204684684146) circle (1.8pt);
\draw[color=black] (5.8,-1.246215860556169) node {$A_4$};
\draw [color=black,fill=white] (6.966065475766681,-5.1972725195092835) circle (1.8pt);
\draw[color=black] (6.8,-6.15) node {$A_3$};
\draw [color=black,fill=white] (-3.2207812957733704,-5.576725264580975) circle (1.8pt);
\draw[color=black] (-3.3,-6.55) node {$A_2$};
\draw[color=black] (-5.5,2.) node {$a_{41}$};
\draw[color=black] (-2.08,11.558430651117668) node {$a_{34}$};
\draw[color=black] (-7.934742679368808,-6.3) node {$a_{23}$};
\draw[color=black] (8.1,11.284045368724657) node {$a_{12}$};
\draw[color=red] (7.2,1.4) node {$g_2$};
\draw[color=darkgreen] (-7.8432809185711365,4.75) node {$h_2$};
\draw[color=red] (-3.3,-11.261278667901063) node {$g_1$};

\draw [color=white,fill=white] (-1.7,-11.2) circle (7pt);
\draw[color=orange] (-1.7,-11.2) node[anchor=center] {$\ell^{(2)}_1$};

\draw [color=white,fill=white] (.75,-10) circle (7pt);
\draw[color=orange] (.75,-10) node[anchor=center] {$\ell^{(2)}_2$};

\draw [color=white,fill=white] (13.6,6.5) circle (8pt);
\draw[color=blue] (13.6,6.5) node[anchor=center] {$\ell^{(1)}_2$};

\draw [color=white,fill=white] (13.6,11.28) circle (9pt);
\draw[color=blue] (13.6,11.28) node[anchor=center] {$\ell^{(1)}_1$};

\draw [color=white,fill=white] (2.6,-6.3) circle (8pt);
\draw [color=white,fill=white] (2.6,-7.6) circle (8pt);
\draw[color=blue] (2.6,-6.3) node[anchor=center] {$\ell^{(3)}_1$};
\draw[color=blue] (2.6,-7.6) node[anchor=center] {$\ell^{(3)}_2$};

\draw [color=white,fill=white] (-8,-3.5) circle (7pt);
\draw [color=white,fill=white] (-8,-4.5) circle (7pt);
\draw[color=orange] (-8,-3.3) node[anchor=center] {$\ell^{(4)}_2$};
\draw[color=orange] (-8,-4.5) node[anchor=center] {$\ell^{(4)}_1$};

\draw[color=darkgreen] (-8.437782363755998,-10.986893385508052) node {$h_1$};
\draw[color=red] (-8.80362940694668,0.8) node {$g_3$};
\draw[color=darkgreen] (1.9888583671784612,-11.261278667901063) node {$h_3$};
\draw[color=red] (6.05,-11.261278667901063) node {$g_4$};
\draw[color=darkgreen] (-7.431702994981619,11.83) node {$h_4$};
\draw [color=black,fill=white] (2.043025144438736,3.054020430181197) circle (1.8pt);
\draw[color=black] (0.7083937160110716,3.55552658132152) node {$A_5$};
\draw [color=black,fill=white] (15.44065391936529,-4.881600170607331) circle (1.8pt);
\draw[color=black] (15.8,-4.081530445283947) node {$A_6$};
\draw [color=black,fill=white] (-6.352233813412808,-7.460143190545852) circle (1.8pt);
\draw [color=black,fill=white] (-2.459938514423172,-6.9996052848852965) circle (1.8pt);
\draw [color=black,fill=white] (-1.5709641052336292,-4.584439654716479) circle (1.8pt);
\draw [color=black,fill=white] (-3.771934781171831,-4.545992895532449) circle (1.8pt);
\draw [color=black,fill=white] (1.6342626852431308,1.5596839057998286) circle (1.8pt);
\draw [color=black,fill=white] (0.9458353289625898,2.2532009215073594) circle (1.8pt);
\draw [color=black,fill=white] (4.6279572941259675,4.9407168475822125) circle (1.8pt);
\draw [color=black,fill=white] (4.364775650496507,11.541777815070624) circle (1.8pt);
\draw [color=black,fill=white] (9.324773300096512,-6.187683949679102) circle (1.8pt);
\draw [color=black,fill=white] (5.107549273231969,-7.088297980117451) circle (1.8pt);
\draw [color=black,fill=white] (7.650417707656965,-4.500949564259514) circle (1.8pt);
\draw [color=black,fill=white] (4.99922787311726,-4.371405736494197) circle (1.8pt);
\draw [color=black,fill=white] (1.3387501317294563,-2.8343866290162523) circle (1.8pt);
\draw [color=black,fill=white] (0.23727500878509566,-12.04376332481183) circle (1.8pt);
\draw [color=black,fill=white] (20.780704581051644,8.859012941751951) circle (1.8pt);
\draw [color=black,fill=white] (-6.431879825794232,0.42846799425110205) circle (1.8pt);
\end{small}
\end{tikzpicture}
\caption{Harmonic pencils in a quadrilateral.}\label{fig:free-quadrilateral}
\end{center}
\end{figure}
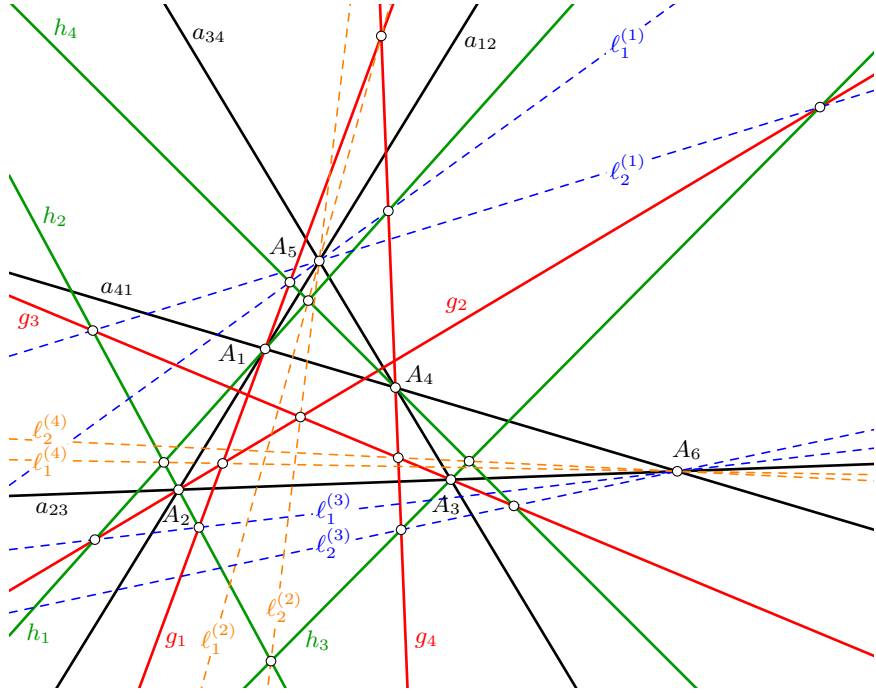
Compared to the situation in Theorem~\ref{steiner-add}, we 
observe that the four lines which carry the quintuples of collinear 
intersection points each split up in two lines (compare the dashed lines in Figures~\ref{fig:steiner-add} 
and~\ref{fig:free-quadrilateral}). 
One can move e.g.~$g_1$ in Proposition~\ref{prop:free-quadrilateral} such that two of the split lines 
collapse. Then, interestingly, not only one of the pairs of lines collapses, 
but all four pairs collapse simultaneously:
\begin{theorem}\label{thm:projective-quadrilateral}
Let  $A_1, A_2, A_3, A_4$ be the vertices of a quadrilateral with sides
$a_{i(i+1)}=A_i\times A_{i+1}$ for $i\in\{1,2,3,4\}$ (where indices are taken cyclically).
The quadrilateral is completed with $A_5=a_{12}\times a_{34}$, $A_6=a_{41}\times a_{23}$.
Assume that in each vertex $A_1,A_2,A_3,A_4$ two lines $g_i, h_i$ 
build together with the sides a harmonic pencil $(a_{(i-1)i} a_{i(i+1)};g_i h_i)$.
%Let $A_1,A_2,A_3,A_4$ be a quadrilateral. For $i,j\in\{1,2,3,4\}$ with $i>j$ 
%let $a_{ij}$ be the line through $A_i$ and $A_j$. For each $i,in\{1,2,3,4\}$
%let $g_i, h_i$ be two lines such that $(a_{(i-1)i} g_i a_{(i+1)i} h_i)$ is a harmonic pencil.
%Let $A_5=a_{12}\times a_{34}$ and $A_6=a_{14}\times a_{23}$. 
%Moreover,
%for $i,j\in\{1,2,3,4\}$ let
%\begin{itemize}
%\item $G_{ij}= g_i\times g_j$ for $i<j$
%\item $H_{ij}= g_i\times h_j$ 
%\end{itemize}
Finally, denote by $\ell^{(i)}$ the following four pairs  of lines (see the dashed lines in Figure~\ref{fig:free-quadrilateral}):
\begin{alignat*}{4}
&\text{line $\ell^{(1)}_1$: } &&\text{$A_5,\  g_1\times h_4,\  h_1\times g_4$\quad and} \quad&& \text{line $\ell^{(1)}_2$: } &&\text{$A_5,\ g_3\times h_2,\ h_3\times g_2$}\\
&\text{line $\ell^{(2)}_1$: } &&\text{$A_5,\  g_1\times g_4,\  h_1\times h_4$\quad and } \quad&& \text{line $\ell^{(2)}_2$: } &&\text{$A_5,\ g_2\times g_3,\ h_2\times h_3$}\\
&\text{line $\ell^{(3)}_1$: } &&\text{$A_6,\  g_1\times h_2,\  h_1\times g_2$\quad and} \quad&& \text{line $\ell^{(3)}_2$: } &&\text{$A_6,\ g_3\times h_4,\ h_3\times g_4$}\\
&\text{line $\ell^{(4)}_1$: } &&\text{$A_6,\  g_1\times g_2,\  h_1\times h_2$\quad and} \quad&& \text{line $\ell^{(4)}_2$: } &&\text{$A_6,\ g_3\times g_4,\ h_3\times h_4$}
%&\text{pair $\ell_1$: } &&\text{line $A_5,\  g_1\times h_4,\  h_1\times g_4$} && \text{ and line $A_5,\ g_3\times h_2,\ h_3\times g_2$}\\
%&\text{pair $\ell_2$: } &&\text{line $A_5,\  g_1\times g_4,\  h_1\times h_4$} && \text{ and line $A_5,\ g_2\times g_3,\ h_2\times h_3$}\\
%&\text{pair $\ell_3$: } &&\text{line $A_6,\  g_1\times h_2,\  h_1\times g_2$} && \text{ and line $A_6,\ g_3\times h_4,\ h_3\times g_4$}\\
%&\text{pair $\ell_4$: } &&\text{line $A_6,\  g_1\times g_2,\  h_1\times h_2$} && \text{ and line $A_6,\ g_3\times g_4,\ h_3\times h_4$}
\end{alignat*}
Then, if the two lines of one of the pairs $\ell^{(i)}$ coincide, then the two lines of all pairs coincide,
and this is the  case if and only if one of the following two equivalent conditions
holds:
%
%Then, the following assertions are equivalent:
%\begin{enumerate}
%\item $A_5,g_1\times g_4,g_2\times g_3$ are collinear
%\item $A_6,g_1\times g_2,g_3\times g_4$ are collinear
%\item One of the four pairs of lines coincide
%\begin{alignat*}{3}
%&\text{pair $1$: } $A_5,\quad && g_1\times h_4,\quad && h_1\times g_4,\
%
%&A_5,\quad && g_1\times h_4,\quad && h_1\times g_4,\qquad\qquad  &&A_6,\quad && g_1\times h_2,\quad && h_1\times g_2,\\
%&A_5,\quad && g_3\times h_2,\quad && h_3\times g_2,\qquad\qquad  &&A_6,\quad && g_3\times h_4,\quad && h_3\times g_4,\\
%&A_5,\quad && g_1\times g_4,\quad && h_1\times h_4,\qquad\qquad  &&A_6,\quad && g_1\times g_2,\quad && h_1\times h_2,\\
%&A_5,\quad && g_2\times g_3,\quad && h_2\times h_3,\qquad\qquad  &&A_6,\quad && g_3\times g_4,\quad && h_3\times h_4\\
%\end{alignat*}
%
%\item $A_5,g_1\times h_4,g_4\times h_1,g_2\times h_3,g_3\times h_2$ are collinear
%\item $A_6,g_1\times h_2,g_2\times h_1,g_3\times h_4,g_4\times h_3$ are collinear
%
%\item $A_5,h_1\times h_4,g_1\times g_4,h_2\times h_3,g_2\times g_3$ are collinear
%\item $A_6,h_3\times h_4,g_3\times g_4,h_1\times h_2,g_1\times g_2$ are collinear
%
%
\begin{enumerate}
\item 
%$
%\frac{A_1B_{123}}{B_{123}A_2}\cdot \frac{A_1B_{124}}{B_{124}A_2}\cdot \frac{A_2B_{231}}{B_{231}A_3}\cdot \frac{A_2B_{234}}{B_{234}A_3}\cdot
%\frac{A_3B_{341}}{B_{341}A_4}\cdot \frac{A_3B_{342}}{B_{342}A_4}\cdot \frac{A_4B_{124}}{B_{124}A_1}\cdot \frac{A_4B_{413}}{B_{413}A_1}=1$\\
$\begin{displaystyle}
\prod_{i=1}^4 \frac{A_iB_{i(i+1)(i+2)}}{B_{i(i+1)(i+2)}A_{i+1}}\cdot \frac{A_iB_{i(i+1)(i+3)}}{B_{i(i+1)(i+3)}A_{i+1}}=1
\end{displaystyle}$\\
where $B_{ijk}=a_{ij}\times g_k$ (see Figure~\ref{fig:theorem8i}).\label{item:8i}
\item 
%$
%\frac{A_1A_2}{A_2B_{123}}\cdot \frac{A_1A_2}{A_2B_{124}}\cdot \frac{A_2A_3}{A_3B_{231}}\cdot \frac{A_2A_3}{A_3B_{234}}$\\
$\begin{displaystyle}
\frac{A_1D_2}{D_2 A_3}\cdot \frac{A_2D_3}{D_3 A_4}\cdot \frac{A_3D_4}{D_4 A_1}\cdot \frac{A_4D_1}{D_1 A_2}=1
\end{displaystyle}$\\
where $D_1=g_{1}\times  (A_2\times A_4)$, $D_3=g_{3}\times ( A_2\times A_4)$, $D_2=g_{2}\times  (A_1\times A_3)$, $D_4=g_{4}\times ( A_1\times A_3)$ (see Figure~\ref{fig:theorem8ii}).\label{diag}
%where $D_i=g_{i}\times  (A_{i-1}\times A_{i+1})$ (see Figure~\ref{fig:theorem8ii}).\label{diag}
\end{enumerate}
\end{theorem}
\begin{proof}
According to Proposition~\ref{prop:free-triangle}(\ref{prop:5ii}), the lines $(a_{12}a_{34};\ell^{(2)}_1\ell^{(1)}_1)$
and $(a_{12}a_{34};\ell^{(2)}_2\ell^{(1)}_2)$ are harmonic. Therefore,  we have $\ell^{(1)}_1=\ell^{(1)}_2$
if and only if $\ell^{(2)}_1=\ell^{(2)}_2$. The same argument yields that $\ell^{(3)}_1=\ell^{(3)}_2$
if and only if $\ell^{(4)}_1=\ell^{(4)}_2$.

Suppose now that $\ell^{(2)}_1=\ell^{(2)}_2$. In order to show that this implies $\ell^{(4)}_1=\ell^{(4)}_2$
we use the Theorem of Menelaos for the following triangles and transversal lines:
\begin{itemize}
\item triangle $A_1A_4A_5$ with line $a_{23}$:
\begin{equation}\label{eq:m1}
\frac{A_4A_6}{A_6A_1}\cdot\frac{A_1A_2}{A_2A_5}\cdot\frac{A_5A_3}{A_3A_4}=-1.
\end{equation}
\item triangle  $A_1A_5(g_1\times g_4)$ with line $g_2$:
\begin{equation}\label{eq:m2}
\frac{A_1(g_1\times g_2)}{(g_1\times g_2)(g_1\times g_4)}\cdot\frac{(g_1\times g_4)(g_2\times g_3)}{(g_2\times g_3)A_5}\cdot\frac{A_5A_2}{A_2A_1}=-1.
\end{equation}
\item triangle  $A_4A_5(g_1\times g_4)$ with line $g_3$:
\begin{equation}\label{eq:m3}
\frac{(g_1\times g_4)(g_3\times g_4)}{(g_3\times g_4)A_4}\cdot\frac{A_4A_3}{A_3A_5}\cdot\frac{A_5(g_2\times g_3)}{(g_2\times g_3)(g_1\times g_4)}=-1.
\end{equation}
\end{itemize}
Multiplication of the equations~(\ref{eq:m1})--(\ref{eq:m3}) yields
\begin{equation}\label{eq:m4}
\frac{A_1(g_1\times g_2)}{(g_1\times g_2)(g_1\times g_4)}\cdot\frac{(g_1\times g_4)(g_3\times g_4)}{(g_3\times g_4)A_4}\cdot\frac{A_4A_6}{A_6A_1}=-1.
\end{equation}
Using Menelaos and~(\ref{eq:m4}) in the triangle $A_1(g_1\times g_4) A_4$ yields that the points
$g_1\times g_2$, $g_3\times g_4$, and $A_6$ are collinear. Hence we have $\ell^{(4)}_1=\ell^{(4)}_2$ as claimed.
The same argument  shows that  $\ell^{(2)}_1=\ell^{(2)}_2$ follows from $\ell^{(4)}_1=\ell^{(4)}_2$.

Now we show that $\ell^{(2)}_1=\ell^{(2)}_2$ is equivalent to condition~(\ref{item:8i}) in the theorem. 
Before we start, we set up some identities which hold in general, i.e., without
the hypothesis $\ell^{(2)}_1=\ell^{(2)}_2$. To this end, we apply
Ceva's Theorem to the following triangles:
\begin{itemize}
\item triangle $A_1A_5A_4$ with point $g_1\times g_4$:
\begin{equation}\label{eq:c1}
\frac{A_1B_{124}}{B_{124}A_5}\cdot\frac{A_5B_{341}}{B_{341}A_4}\cdot\frac{A_4C_{41}}{C_{41}A_1}=1,
\end{equation}
where $C_{41}=a_{41}\times \ell^{(2)}_1$.
\item triangle $A_2A_5A_3$ with point $g_2\times g_3$:
\begin{equation}\label{eq:c2}
\frac{A_5B_{123}}{B_{123}A_2}\cdot\frac{A_2C_{23}}{C_{23}A_3}\cdot\frac{A_3B_{342}}{B_{342}A_5}=1,
\end{equation}
where $C_{23}=a_{23}\times \ell^{(2)}_2$.
\end{itemize}
Now we employ Menelaos once more as follows:
\begin{itemize}
\item triangle $A_1A_5A_4$ with transversal line $g_2$:
\begin{equation}\label{eq:mm1}
\frac{A_1A_2}{A_2A_5}\cdot\frac{A_5B_{342}}{B_{342}A_4}\cdot\frac{A_4B_{412}}{B_{412}A_1}=-1.
\end{equation}
\item triangle $A_1A_5A_4$ with transversal line $g_3$:
\begin{equation}\label{eq:mm2}
\frac{A_1B_{123}}{B_{123}A_5}\cdot\frac{A_5A_3}{A_3A_4}\cdot\frac{A_4B_{413}}{B_{413}A_1}=-1.
\end{equation}
\item triangle $A_2A_5A_3$ with transversal line $g_1$:
\begin{equation}\label{eq:mm3}
\frac{A_5A_1}{A_1A_2}\cdot\frac{A_2B_{231}}{B_{231}A_3}\cdot\frac{A_3B_{341}}{B_{341}A_5}=-1.
\end{equation}
\item triangle $A_2A_5A_3$ with transversal line $g_4$:
\begin{equation}\label{eq:mm4}
\frac{A_5B_{124}}{B_{124}A_2}\cdot\frac{A_2B_{234}}{B_{234}A_3}\cdot\frac{A_3A_4}{A_4A_5}=-1.
\end{equation}
\end{itemize}
The product of the six equations~(\ref{eq:c1})--(\ref{eq:mm4}) simplifies to
$$
\zeta\cdot\frac{A_1A_5}{A_5A_2}\cdot\frac{A_2C_{23}}{C_{23}A_3}\cdot\frac{A_3A_5}{A_5A_4}\cdot\frac{A_4C_{41}}{C_{41}A_1}=1,
$$
where $\zeta$ is the expression (\ref{item:8i}) in the theorem. Hence, $\zeta=1$ if and only if 
\begin{equation}\label{eq:zeta}
\frac{A_1A_5}{A_5A_2}\cdot\frac{A_2C_{23}}{C_{23}A_3}\cdot\frac{A_3A_5}{A_5A_4}\cdot\frac{A_4C_{41}}{C_{41}A_1}=1.
\end{equation}
So, we are done if we can show that~(\ref{eq:zeta}) is equivalent to the condition $\ell^{(2)}_1=\ell^{(2)}_2$.
Indeed, define the points $C_{23}'=\ell^{(2)}_2 \times (A_2\times A_4)$ and $C_{41}'=\ell^{(2)}_1 \times (A_2\times A_4)$
(see Figure~\ref{fig:h}).
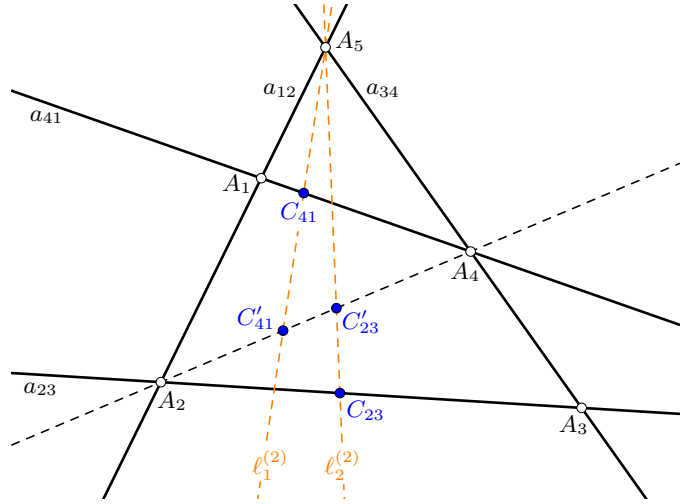
\begin{figure}[h] 
\begin{center}
\begin{tikzpicture}[line cap=round,line join=round,x=18,y=18]
\clip(-4.879662428297845,-7.1) rectangle (9.004205261550823,3.2672944728467885);
\draw [line width=1pt,domain=-4.879662428297845:9.004205261550823] plot(\x,{(-1.1970170352558647-1.538057814938893*\x)/4.353485679572828});
\draw [line width=1pt,domain=-4.879662428297845:9.004205261550823] plot(\x,{(--10.829964661987475-3.2766665237551305*\x)/2.3080392339857827});
\draw [line width=1pt,domain=-4.879662428297845:9.004205261550823] plot(\x,{(--41.70287412723992--0.5394449886943722*\x)/-8.747027035509666});
\draw [line width=1pt,domain=-4.879662428297845:9.004205261550823] plot(\x,{(-2.099801019684186--4.275279349999652*\x)/2.0855021219510554});
\draw [line width=.6pt,dash pattern=on 3pt off 3pt,color=orange,domain=-4.879662428297845:9.004205261550823] plot(\x,{(--5.795998488763779-4.488143425552458*\x)/-0.6676999568423239});
\draw [line width=.6pt,dash pattern=on 3pt off 3pt,color=orange,domain=-4.879662428297845:9.004205261550823] plot(\x,{(-19.55694684829457--11.24001477219463*\x)/-0.4637361067373289});
\draw [line width=.6pt,dash pattern=on 3pt off 3pt,domain=-4.879662428297845:9.004205261550823] plot(\x,{(-25.116808318075428--2.7372215350607583*\x)/6.438987801523883});
\begin{small}
\draw [color=black,fill=white] (0.3045405622080699,-0.3825481808152599) circle (1.8pt);
\draw[color=black] (-0.2,-0.5) node {$A_1$};
\draw [color=black,fill=white] (4.658026241780898,-1.9206059957541528) circle (1.8pt);
\draw[color=black] (4.55,-2.35) node {$A_4$};
\draw [color=black,fill=white] (6.966065475766681,-5.1972725195092835) circle (1.8pt);
\draw[color=black] (6.8,-5.55) node {$A_3$};
\draw [color=black,fill=white] (-1.7809615597429853,-4.657827530814911) circle (1.8pt);
\draw[color=black] (-1.55,-5.05) node {$A_2$};
\draw[color=black] (-4.154955048882667,0.9) node {$a_{41}$};
\draw[color=black] (2.8378354542462403,1.46) node {$a_{34}$};
\draw[color=black] (-4.2820966943941015,-4.8) node {$a_{23}$};
\draw[color=black] (0.7,1.46) node {$a_{12}$};
\draw [color=black,fill=white] (1.6425558445452941,2.360383147976042) circle (1.8pt);
\draw[color=black] (2.15,2.48) node {$A_5$};
\draw [color=black,fill=white] (15.405945915269998,-5.717775221717341) circle (1.8pt);
\draw[color=black] (16.87427311870863,-9.167158458171551) node {$A_6$};

\draw [color=white,fill=white] (1,-1.16) circle (6pt);
\draw [fill=blue] (1.1880537501369801,-0.6946875705625972) circle (1.8pt);
\draw[color=blue] (1.1,-1.1) node[anchor=center] {$C_{41}$};

\draw [fill=blue] (1.9415821758708158,-4.88740354049458) circle (1.8pt);
\draw [fill=blue] (0.7590973373263732,-3.57804540857895) circle (1.8pt);
\draw[color=blue] (.2,-3.3) node {$C_{41}'$};
\draw [fill=blue] (1.8681105884318072,-3.1066025285543737) circle (1.8pt);
\draw[color=blue] (2.48,-5.25) node[anchor=center]  {$C_{23}$};

\draw[color=blue] (2.3801255304050755,-3.394927751952404) node {$C_{23}'$};
\draw [color=white,fill=white] (2,-6.4) circle (7pt);
\draw [color=white,fill=white] (.5,-6.4) circle (7pt);

\draw[color=orange] (2,-6.4) node[anchor=center] {$\ell^{(2)}_2$};
\draw[color=orange] (.5,-6.4) node[anchor=center] {$\ell^{(2)}_1$};

\end{small}
\end{tikzpicture}
\caption{Equivalence of (\ref{eq:zeta}) and $\ell^{(2)}_1=\ell^{(2)}_2$.}\label{fig:h}
\end{center}
\end{figure}
We use again the Theorem of Menelaos as follows:
\begin{itemize}
\item triangle $A_1A_2A_4$ with transversal line $\ell^{(2)}_1$:
\begin{equation}\label{eq:mm33}
\frac{A_1A_5}{A_5A_2}\cdot\frac{A_2C_{41}'}{C_{41}'A_4}\cdot\frac{A_4C_{41}}{C_{41}A_1}=-1.
\end{equation}
\item triangle $A_2A_3A_4$ with transversal line $\ell^{(2)}_2$:
\begin{equation}\label{eq:mm44}
\frac{A_2C_{23}}{C_{23}A_3}\cdot\frac{A_3A_5}{A_5A_4}\cdot\frac{A_4C_{23}'}{C_{23}'A_2}=-1.
\end{equation}
\end{itemize}
The product of~(\ref{eq:mm33}) and~(\ref{eq:mm44}) is 
$$
\frac{A_1A_5}{A_5A_2}\cdot\frac{A_2C_{23}}{C_{23}A_3}\cdot\frac{A_3A_5}{A_5A_4}\cdot
\frac{A_4C_{41}}{C_{41}A_1}\cdot\frac{A_2C_{41}'}{C_{41}'A_4}\cdot\frac{A_4C_{23}'}{C_{23}'A_2}=1.
$$
Hence, equation~(\ref{eq:zeta}) holds if and only if 
%\begin{equation}\label{eq:eq}
$$\frac{A_2C_{41}'}{C_{41}'A_4}\cdot\frac{A_4C_{23}'}{C_{23}'A_2}=1,$$
%\end{equation}
or  equivalently
$$
\frac{A_2C_{41}'}{C_{41}'A_4}=\frac{A_2C_{23}'}{C_{23}'A_4}.
$$
But this is indeed true if and only if $\ell^{(2)}_1=\ell^{(2)}_2$, as claimed.

It remains to show that the conditions~(\ref{item:8i}) and~(\ref{diag}) in the theorem are equivalent.
We use again Menelaos (see Figure~\ref{fig:theorem8ii}):
\begin{itemize}
\item triangle $A_4A_2A_3$ with transversal line $g_1$:
\begin{equation}\label{eq:q1}
\frac{A_4D_1}{D_1A_2}\cdot\frac{A_2B_{231}}{B_{231}A_3}\cdot\frac{A_3B_{341}}{B_{341}A_4}=-1.
\end{equation}
\item triangle $A_1A_3A_4$ with transversal line $g_2$:
\begin{equation}\label{eq:q2}
\frac{A_1D_2}{D_2A_3}\cdot\frac{A_3B_{342}}{B_{342}A_4}\cdot\frac{A_4B_{412}}{B_{412}A_1}=-1.
\end{equation}
\item triangle $A_2A_4A_1$ with transversal line $g_3$:
\begin{equation}\label{eq:q3}
\frac{A_2D_3}{D_3A_4}\cdot\frac{A_4B_{413}}{B_{413}A_1}\cdot\frac{A_1B_{123}}{B_{123}A_2}=-1.
\end{equation}
\item triangle $A_3A_1A_2$ with transversal line $g_4$:
\begin{equation}\label{eq:q4}
\frac{A_3D_4}{D_4A_1}\cdot\frac{A_1B_{124}}{B_{124}A_2}\cdot\frac{A_2B_{234}}{B_{234}A_3}=-1.
\end{equation}
\end{itemize}
The product of (\ref{eq:q1})--(\ref{eq:q4}) is
$$
\zeta\cdot\frac{A_4D_1}{D_1A_2}\cdot\frac{A_1D_2}{D_2A_3}\cdot\frac{A_2D_3}{D_3A_4}\cdot\frac{A_3D_4}{D_4A_1}=1
$$
which shows that indeed~(\ref{item:8i}) and~(\ref{diag}) in the theorem are equivalent.
\end{proof}
{\bf Remark.} Note that if we fix three of the four lines  $g_1,\ldots,g_4$, then
there exists a unique position for the fourth line such that the criterion~(\ref{diag}) 
in Theorem~\ref{thm:projective-quadrilateral} is satisfied.

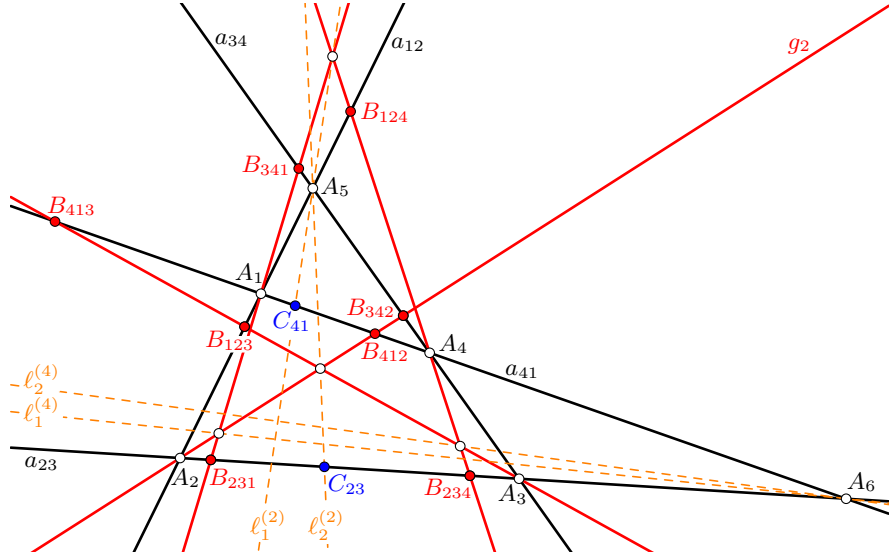
\begin{figure} 
\begin{center}
\begin{tikzpicture}[line cap=round,line join=round,x=14.5,y=14.5]
\clip(-6.151078883412192,-7.1) rectangle (16.50556234672547,7.1578288254967);
\draw [line width=1pt,domain=-6.151078883412192:16.50556234672547] plot(\x,{(-1.1970170352558647-1.538057814938893*\x)/4.353485679572828});
\draw [line width=1pt,domain=-6.151078883412192:16.50556234672547] plot(\x,{(--10.829964661987475-3.2766665237551305*\x)/2.3080392339857827});
\draw [line width=1pt,domain=-6.151078883412192:16.50556234672547] plot(\x,{(--41.70287412723992--0.5394449886943722*\x)/-8.747027035509666});
\draw [line width=1pt,domain=-6.151078883412192:16.50556234672547] plot(\x,{(-2.099801019684186--4.275279349999652*\x)/2.0855021219510554});
\draw [line width=1pt,color=red,domain=-6.151078883412192:16.50556234672547] plot(\x,{(-20.220512895729883--3.704511034898051*\x)/5.757642263288148});
\draw [line width=1pt,color=red,domain=-6.151078883412192:16.50556234672547] plot(\x,{(--1.8120911868802534-4.324056740445125*\x)/-1.294583377890019});
\draw [line width=1pt,color=red,domain=-6.151078883412192:16.50556234672547] plot(\x,{(--9.23649089029152--3.9564840943584376*\x)/-7.0801786912832885});
\draw [line width=1pt,color=red,domain=-6.151078883412192:16.50556234672547] plot(\x,{(--12.902923592384992-3.198395400077251*\x)/1.0388836217199673});
\draw [line width=.6pt,dash pattern=on 3pt off 3pt,color=orange,domain=-6.151078883412192:16.50556234672547] plot(\x,{(-5.561054538944495-0.1426478411300276*\x)/1.3569402719530843});
\draw [line width=.6pt,dash pattern=on 3pt off 3pt,color=orange,domain=-6.151078883412192:16.50556234672547] plot(\x,{(--5.795998488763773-4.488143425552457*\x)/-0.6676999568423239});
\draw [line width=.6pt,dash pattern=on 3pt off 3pt,color=orange,domain=-6.151078883412192:16.50556234672547] plot(\x,{(--29.18353309665426--1.1151016985373632*\x)/-8.10852608865262});
\draw [line width=.6pt,dash pattern=on 3pt off 3pt,color=orange,domain=-6.151078883412192:16.50556234672547] plot(\x,{(-19.556946848294572--11.240014772194627*\x)/-0.4637361067373291});
\begin{small}
\draw [color=black,fill=white] (0.3045405622080699,-0.3825481808152599) circle (1.8pt);
\draw[color=black] (.01,.15) node {$A_1$};
\draw [color=black,fill=white] (4.658026241780898,-1.9206059957541528) circle (1.8pt);
\draw[color=black] (5.27,-1.6848849169069489) node {$A_4$};
\draw [color=black,fill=white] (6.966065475766681,-5.1972725195092835) circle (1.8pt);
\draw[color=black] (6.8,-5.6) node {$A_3$};
\draw [color=black,fill=white] (-1.7809615597429853,-4.657827530814911) circle (1.8pt);
\draw[color=black] (-1.65,-5.18) node {$A_2$};
\draw[color=black] (7.05,-2.4) node {$a_{41}$};
\draw[color=black] (-0.45,6.217473991703563) node {$a_{34}$};
\draw[color=black] (-5.355085446081155,-4.8) node {$a_{23}$};
\draw[color=black] (4.1,6.11) node {$a_{12}$};
\draw [color=black] (3.976680703545163,-0.9533164959168606) circle (1.8pt);
\draw[color=red] (14.2,6.07) node {$g_2$};
\draw [color=black,fill=white] (-0.9900428156819493,-4.706604921260385) circle (1.8pt);
\draw[color=red] (-2.15,-8.412899467581532) node {$g_1$};
\draw [color=black,fill=white] (-0.11411321551660736,-1.2407884251508459) circle (1.8pt);
\draw[color=red] (12.1,-8.412899467581532) node {$g_3$};
\draw [color=black,fill=white] (5.696909863500865,-5.119001395831404) circle (1.8pt);
\draw[color=red] (7.24,-8.412899467581532) node {$g_4$};
\draw [color=black,fill=white] (1.6425558445452941,2.360383147976042) circle (1.8pt);
\draw[color=black] (2.22,2.4) node {$A_5$};
\draw [color=black,fill=white] (15.405945915269998,-5.717775221717341) circle (1.8pt);
\draw[color=black] (15.80628329641258,-5.276624105521643) node {$A_6$};
\draw [color=black,fill=white] (-0.7832496491003572,-4.01589206250847) circle (1.8pt);

\draw [color=white,fill=white] (-5.35,-2.6) circle (8pt);
\draw [color=white,fill=white] (-5.35,-3.44) circle (8pt);
\draw[color=orange] (-5.35,-3.46) node {$\ell^{(4)}_1$};
\draw[color=orange] (-5.35,-2.6) node {$\ell^{(4)}_2$};

\draw [color=black,fill=white] (2.1530097578377916,5.791550622998449) circle (1.8pt);

\draw [color=black,fill=white] (5.4465213499817695,-4.348134040596951) circle (1.8pt);
\draw [color=black,fill=white] (1.8360940174253262,-2.3305863291694253) circle (1.8pt);
\draw [color=black,fill=white] (2.2998301241626553,-13.570601101364051) circle (1.8pt);

\draw [color=white,fill=white] (2,-6.4) circle (7pt);
\draw [color=white,fill=white] (.5,-6.4) circle (7pt);

\draw[color=orange] (2,-6.4) node[anchor=center] {$\ell^{(2)}_2$};
\draw[color=orange] (.5,-6.4) node[anchor=center] {$\ell^{(2)}_1$};

\draw [fill=red] (-0.12,-1.2407884251508456) circle (1.8pt);
\draw [fill=white,color=white] (-.5-.5,-1.6-.3) rectangle (-.5+.7,-1.6+.1); 
\draw[color=red,fill=white] (-.5,-1.6) node {$B_{123}$};
\draw [fill=red] (5.696909863500866,-5.119001395831404) circle (1.8pt);
\draw[color=red] (5.14,-5.45) node {$B_{234}$};
\draw [fill=red] (1.2798942618834455,2.8752448755680837) circle (1.8pt);
\draw[color=red] (0.44,2.962154523619345) node {$B_{341}$};
\draw [fill=red] (3.2477012317588874,-1.422347459279314) circle (1.8pt);
\draw[color=red] (3.5,-1.9) node {$B_{412}$};
\draw [fill=red] (2.6179899551014887,4.360023074616235) circle (1.8pt);
\draw[color=red] (3.5,4.33) node {$B_{124}$};
\draw [fill=red] (-0.9900428156819492,-4.706604921260386) circle (1.8pt);
\draw[color=red] (-0.4,-5.19) node {$B_{231}$};
\draw [fill=red] (3.9766807035451635,-0.9533164959168607) circle (1.8pt);
\draw[color=red] (3.142975403473684,-0.75) node {$B_{342}$};
\draw [fill=red] (-5.009782974485891,1.494967320172414) circle (1.8pt);
\draw[color=red] (-4.587236643621545,1.9) node {$B_{413}$};

\draw [color=white,fill=white] (1,-1.16) circle (6pt);
\draw [fill=blue] (1.1880537501369801,-0.6946875705625972) circle (1.8pt);
\draw[color=blue] (1.1,-1.1) node[anchor=center] {$C_{41}$};
\draw [fill=blue] (1.941582175870817,-4.887403540494581) circle (1.8pt);
\draw[color=blue] (2.53,-5.35) node[anchor=center]  {$C_{23}$};
\end{small}
\end{tikzpicture}
\caption{Lines and points in Theorem~\ref{thm:projective-quadrilateral}(\ref{item:8i}).}\label{fig:theorem8i}
\end{center}
\end{figure}
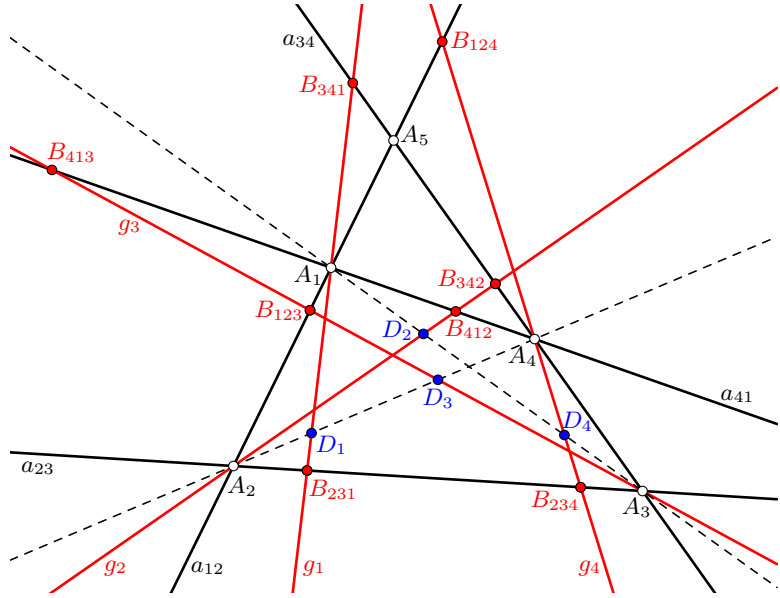
\begin{figure}
\begin{center}
\begin{tikzpicture}[line cap=round,line join=round,x=17.5,y=17.5]
\clip(-6.557932149048781,-7.387175421011461) rectangle (9.843340121926289,5.301560801029748);
\draw [line width=1pt,domain=-6.557932149048781:9.843340121926289] plot(\x,{(-1.1970170352558647-1.538057814938893*\x)/4.353485679572828});
\draw [line width=1pt,domain=-6.557932149048781:9.843340121926289] plot(\x,{(--10.829964661987475-3.2766665237551305*\x)/2.3080392339857827});
\draw [line width=1pt,domain=-6.557932149048781:9.843340121926289] plot(\x,{(--41.70287412723992--0.5394449886943722*\x)/-8.747027035509666});
\draw [line width=1pt,domain=-6.557932149048781:9.843340121926289] plot(\x,{(-2.099801019684186--4.275279349999652*\x)/2.0855021219510554});
\draw [line width=1pt,color=red,domain=-6.557932149048781:9.843340121926289] plot(\x,{(-19.092385335331244--3.9273787031091154*\x)/5.600657315784682});
\draw [line width=1pt,color=red,domain=-6.557932149048781:9.843340121926289] plot(\x,{(--1.5269966897605125-4.372390735374463*\x)/-0.5108542317460205});
\draw [line width=1pt,color=red,domain=-6.557932149048781:9.843340121926289] plot(\x,{(--9.800562851743056--3.897116349678612*\x)/-7.109138566736862});
\draw [line width=1pt,color=red,domain=-6.557932149048781:9.843340121926289] plot(\x,{(--12.992750768026935-3.195003690829563*\x)/0.983887518117422});
\draw [line width=.6pt,dash pattern=on 3pt off 3pt,domain=-6.557932149048781:9.843340121926289] plot(\x,{(-25.116808318075428--2.7372215350607583*\x)/6.438987801523883});
\draw [line width=.6pt,dash pattern=on 3pt off 3pt,domain=-6.557932149048781:9.843340121926289] plot(\x,{(-1.0820753801546223-4.814724338694024*\x)/6.661524913558611});
\begin{small}
\draw [color=black,fill=white] (0.3045405622080699,-0.3825481808152599) circle (1.8pt);
\draw[color=black] (-0.18,-0.55) node {$A_1$};
\draw [color=black,fill=white] (4.658026241780898,-1.9206059957541528) circle (1.8pt);
\draw[color=black] (4.4,-2.3) node {$A_4$};
\draw [color=black,fill=white] (6.966065475766681,-5.1972725195092835) circle (1.8pt);
\draw[color=black] (6.84,-5.55) node {$A_3$};
\draw [color=black,fill=white] (-1.7809615597429853,-4.657827530814911) circle (1.8pt);
\draw[color=black] (-1.6,-5.1) node {$A_2$};
\draw[color=black] (9,-3.15) node {$a_{41}$};
\draw[color=black] (-0.37,4.51) node {$a_{34}$};
\draw[color=black] (-5.960366415145039,-4.7) node {$a_{23}$};
\draw[color=black] (-2.33,-6.853180509863434) node {$a_{12}$};
\draw [color=black,fill=white] (3.819695756041697,-0.7304488277057958) circle (1.8pt);
\draw[color=red] (-4.3,-6.853180509863434) node {$g_2$};
\draw [color=black] (-0.20631366953795055,-4.754938916189722) circle (1.8pt);
\draw[color=red] (-.05,-6.853180509863434) node {$g_1$};
\draw [color=black,fill=white] (-0.14307309097018092,-1.3001561698306716) circle (1.8pt);
\draw[color=red] (-4.,0.5) node {$g_3$};
\draw [color=black,fill=white] (5.64191375989832,-5.115609686583716) circle (1.8pt);
\draw[color=red] (5.84,-6.853180509863434) node {$g_4$};
\draw [color=black,fill=white] (1.6425558445452941,2.360383147976042) circle (1.8pt);
\draw[color=black] (2.1,2.47) node {$A_5$};
\draw [color=black,fill=white] (15.405945915269998,-5.717775221717341) circle (1.8pt);
\draw[color=black] (-6.265506364372482,5.581272421154905) node {$A_6$};
\draw [fill=red] (-0.14307309097018092,-1.3001561698306716) circle (1.8pt);
\draw[color=red] (-.8,-1.3) node {$B_{123}$};
\draw [fill=red] (5.64191375989832,-5.115609686583715) circle (1.8pt);
\draw[color=red] (5.1,-5.4) node {$B_{234}$};
\draw [fill=red] (0.7697815094158698,3.5994394017003852) circle (1.8pt);
\draw[color=red] (0.14,3.547006092971945) node {$B_{341}$};
\draw [fill=red] (2.9719421533457675,-1.3249235932531227) circle (1.8pt);
\draw[color=red] (3.26,-1.7) node {$B_{412}$};
\draw [fill=red] (2.6829348253870227,4.49316005870158) circle (1.8pt);
\draw[color=red] (3.38,4.51) node {$B_{124}$};
\draw [fill=red] (-0.20631366953795052,-4.7549389161897215) circle (1.8pt);
\draw[color=red] (0.35,-5.16) node {$B_{231}$};
\draw [fill=red] (3.8196957560416975,-0.7304488277057959) circle (1.8pt);
\draw[color=red] (3.1,-0.6232398798031216) node {$B_{342}$};
\draw [fill=red] (-5.66281985156279,1.7256809474031736) circle (1.8pt);
\draw[color=red] (-5.248373200281004,2.1) node {$B_{413}$};
\draw [fill=blue] (2.5913800943825227,-2.7991397831097373) circle (1.8pt);
\draw[color=blue] (2.6,-3.2) node {$D_3$};
\draw [fill=blue] (-0.11207865409432947,-3.9483833806403825) circle (1.8pt);
\draw[color=blue] (0.3204308731198336,-4.2) node {$D_1$};
\draw [fill=blue] (2.279855748935079,-1.8102390218070805) circle (1.8pt);
\draw[color=blue] (1.72,-1.73) node {$D_2$};
\draw [fill=blue] (5.29516458311534,-3.9896020086597863) circle (1.8pt);
\draw[color=blue] (5.57,-3.7000677011798477) node {$D_4$};
\end{small}
\end{tikzpicture}

\caption{Lines and points in Theorem~\ref{thm:projective-quadrilateral}(\ref{diag}).}\label{fig:theorem8ii}
\end{center}
\end{figure}
%

%%%%%%%%%%%%%%%%%%%%%%%%%%%%%%%%%%%%%%%%%
\section{%Insights on related configurations
%Applications to classical configurations
Broadening classical results}\label{sec:constellations}

Condition (\ref{diag}) in Theorem~\ref{thm:projective-quadrilateral} is particularly interesting as it can be framed in various ways that give further insights to specific constellations. First, we can derive a statement about quadruples of points with the same cross-ratio:
\begin{corollary}\label{cor:crossratio}
Let $A_1,A_2,A_3,A_4$ and $B_1,B_2,B_3,B_4$ be two quadruples of collinear points. Then the following statements are equivalent:
\begin{enumerate}
\item $(A_1\times B_1)\times (A_2\times B_2),(A_3\times B_3)\times (A_4\times B_4),(A_1\times B_3)\times (A_4\times B_2),(A_1\times B_4)\times (A_3\times B_2),(A_2\times B_3)\times (A_4\times B_1),(A_2\times B_4)\times (A_3\times B_1)$ are collinear.\label{cr1}
\item $(A_1\times B_1)\times (A_3\times B_3),(A_2\times B_2)\times (A_4\times B_4),(A_1\times B_2)\times (A_4\times B_3),(A_1\times B_4)\times (A_2\times B_3),(A_2\times B_1)\times (A_3\times B_4),(A_3\times B_2)\times (A_4\times B_1)$ are collinear.\label{cr2}
\item $(A_1\times B_1)\times (A_4\times B_4),(A_2\times B_2)\times (A_3\times B_3),(A_1\times B_2)\times (A_3\times B_4),(A_1\times B_3)\times (A_2\times B_4),(A_2\times B_1)\times (A_4\times B_3),(A_3\times B_1)\times (A_4\times B_2)$ are collinear.\label{cr3}
\item $(A_1\times B_2)\times (A_2\times B_1),(A_2\times B_3)\times (A_3\times B_2),(A_3\times B_4)\times (A_4\times B_3),(A_4\times B_1)\times (A_1\times B_4),(A_1\times B_3)\times (A_3\times B_1),(A_2\times B_4)\times (A_4\times B_2)$ are collinear.\label{cr4}
\item $\begin{displaystyle}
\operatorname{CR}(A_1,A_2;A_3,A_4)=\operatorname{CR}(B_1,B_2;B_3,B_4)
\end{displaystyle}$\label{cr5}
\end{enumerate}
\end{corollary}
\begin{proof}
By applying Theorem~\ref{thm:projective-quadrilateral}\,(\ref{diag}) to the quadrilateral $A_1B_1A_2B_2$ and the lines $g_1=A_1\times B_3$, $g_2=B_1\times A_3$, $g_3=A_2\times B_4$, $g_4=B_2\times A_4$ 
%\begin{alignat*}{4}
%& A_1\rightarrow A_1,\quad && A_3\rightarrow A_2,\quad && D_2\rightarrow A_3,\quad && D_4\rightarrow A_4,\\
%& A_2\rightarrow B_1,\quad && A_4\rightarrow B_2,\quad && D_1\rightarrow B_3,\quad && D_3\rightarrow B_4
%\end{alignat*}
we get that
\begin{equation}\label{eq:crossratio1}
\frac{A_1 A_3}{A_3 A_2}\cdot \frac{B_1 B_4}{B_4 B_2}\cdot \frac{A_2 A_4}{A_4 A_1}\cdot \frac{B_2 B_3}{B_3 B_1}=1,
\end{equation}
or equivalently, $\operatorname{CR}(A_1,A_2;A_3,A_4)=\operatorname{CR}(B_1,B_2;B_3,B_4)$, is equivalent to the collinearity of the points
$$(A_1\times B_1)\times (A_2\times B_2),\quad (A_1\times B_3)\times(A_4\times B_2),\quad (A_2\times B_4)\times(A_3\times B_1)$$
(see pair $\ell^{(2)}$ in Theorem~\ref{thm:projective-quadrilateral}).

Analogously, $\operatorname{CR}(A_3,A_4;A_1,A_2)=\operatorname{CR}(B_3,B_4;B_1,B_2)$ is equivalent to the col\-lin\-e\-ar\-i\-ty of the points
$$(A_3\times B_3)\times (A_4\times B_4),\quad (A_3\times B_1)\times(A_2\times B_4),\quad (A_4\times B_2)\times(A_1\times B_3),$$
$\operatorname{CR}(A_1,A_2;A_4,A_3)=\operatorname{CR}(B_1,B_2;B_4,B_3)$ is equivalent to the collinearity of the points
$$(A_1\times B_1)\times (A_2\times B_2),\quad (A_1\times B_4)\times(A_3\times B_2),\quad (A_2\times B_3)\times(A_4\times B_1)$$
and $\operatorname{CR}(A_3,A_4;A_2,A_1)=\operatorname{CR}(B_3,B_4;B_2,B_1)$ is equivalent to the collinearity of the points
$$(A_3\times B_3)\times (A_4\times B_4),\quad (A_3\times B_2)\times(A_1\times B_4),\quad (A_4\times B_1)\times(A_2\times B_3).$$
Since all the cross-ratio equations are in fact equivalent and because the triples of points have sufficiently many common points, it follows that
\begin{equation}\label{eq:crossratio2}
\operatorname{CR}(A_1,A_2;A_3,A_4)=\operatorname{CR}(B_1,B_2;B_3,B_4)
\end{equation}
is equivalent to the collinearity of all those points, which proves the equivalence of (\ref{cr1}) and (\ref{cr5}).

Now (\ref{cr2}), (\ref{cr3}) and (\ref{cr4}) follow by noting that
\begin{equation}\label{eq:crossratio5}
\operatorname{CR}(A_1,A_3;A_2,A_4)=\operatorname{CR}(B_1,B_3;B_2,B_4),
\end{equation}
\begin{equation}\label{eq:crossratio3}
\operatorname{CR}(A_1,A_4;A_2,A_3)=\operatorname{CR}(B_1,B_4;B_2,B_3)
\end{equation}
and
\begin{equation}\label{eq:crossratio4}
\operatorname{CR}(A_1,A_2;A_3,A_4)=\operatorname{CR}(B_2,B_1;B_4,B_3)
\end{equation}
are all equivalent to (\ref{eq:crossratio2}) and taking the corresponding intersection points.
\end{proof}
Note that the ordinary Pappus Theorem follows from Corollary~\ref{cor:crossratio} by
letting $A_1$ and $A_3$ as well as $B_1$ and $B_3$ coincide:
This implies $\operatorname{CR}(A_1,A_2;A_3,A_4)=0=\operatorname{CR}(B_1,B_2;B_3,B_4)$ hence~(\ref{cr4}) in Corollary~\ref{cor:crossratio}
will always be true regardless of the constellation of $A_1$, $A_2$, $A_4$, $B_1$, $B_2$, $B_4$.

%Note that when $A_1$ and $A_3$ as well as $B_1$ and $B_3$ coincide, Corollary~\ref{cor:crossratio}\,(\ref{cr4}) gives us Pappus' Theorem as special case because $\operatorname{CR}(A_1,A_2;A_3,A_4)=0=\operatorname{CR}(B_1,B_2;B_3,B_4)$.
Interestingly, Corollary~\ref{cor:crossratio}\,(\ref{cr4}) is a generalization of Pappus' 
Theorem with four points on each of two lines:
% that states that all intersection points of the form $(A_i\times B_j)\times (A_j\times A_i)$ are collinear 
%if and only if the points on the two initial lines have the same cross-ratio. 
\begin{corollary}\label{cor-pappus}
Let $A_1,\ldots,A_4$ and $B_1,\ldots,B_4$ be collinear points, respectively.
By avoiding two of the points $A_i,B_i$, the remaining six points
form a hexagon $H_i$ with sides $A_j\times B_k$, $i\neq j\neq k\neq i$.
The intersections of opposite sides of the hexagon $H_i$ lie on a Pappus line $p_i$.
Then, two of these Pappus lines (and hence all four) coincide if and only if 
$\operatorname{CR}(A_1,A_2;A_3,A_4)=\operatorname{CR}(B_1,B_2;B_3,B_4)$.
\end{corollary}
Figure~\ref{fig-pappus} illustrates the previous corollary.
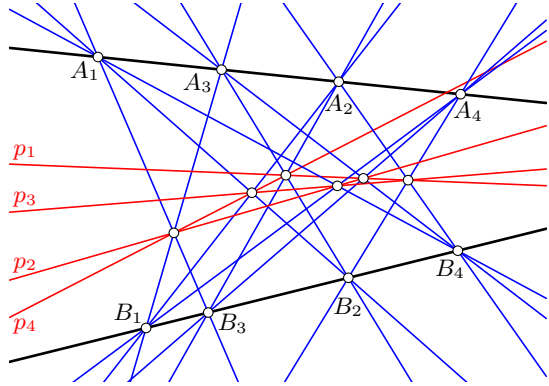
\begin{figure}[h]
\begin{center}
\definecolor{red}{rgb}{1.,0.,0.}
\begin{tikzpicture}[line cap=round,line join=round,>=triangle 45,x=18,y=18]
\clip(-3.06,-4.06) rectangle (8.1,3.86);
\draw [line width=1pt,domain=-3.06:8.1] plot(\x,{(--19.708-0.78*\x)/7.54});
\draw [line width=1pt,domain=-3.06:8.1] plot(\x,{(-18.6948--1.62*\x)/6.48});
\draw [line width=.6pt,color=blue,domain=-3.06:8.1] plot(\x,{(-3.4208919169108185--5.414270533627416*\x)/1.568718174934983});
\draw [line width=.6pt,color=blue,domain=-3.06:8.1] plot(\x,{(-0.2538356422682355-5.356866637296243*\x)/2.2925334508150295});
\draw [line width=.6pt,color=blue,domain=-3.06:8.1] plot(\x,{(-10.639122278039304--5.1622373506271195*\x)/4.005038943937847});
\draw [line width=.6pt,color=blue,domain=-3.06:8.1] plot(\x,{(-8.615635911636982--4.628667002328655*\x)/-5.20533199068538});
\draw [line width=.6pt,color=blue,domain=-3.06:8.1] plot(\x,{(--12.408069023716154-4.362937535956071*\x)/2.6366138157503967});
\draw [line width=.6pt,color=blue,domain=-3.06:8.1] plot(\x,{(-12.288366027456108--4.839103987923362*\x)/2.712505493122818});
\draw [line width=.6pt,color=red,domain=-3.06:8.1] plot(\x,{(-2.6354370420015356--1.207852418136228*\x)/2.328328221666017});
\draw [line width=.6pt,color=blue,domain=-3.06:8.1] plot(\x,{(-18.640762711309712--4.576866637296243*\x)/5.247466549184971});
\draw [line width=.6pt,color=blue,domain=-3.06:8.1] plot(\x,{(-17.269241531421798--3.794270533627416*\x)/-4.911281825065017});
\draw [line width=.6pt,color=blue,domain=-3.06:8.1] plot(\x,{(-19.747626156460445--3.848667002328655*\x)/2.3346680093146204});
\draw [line width=.6pt,color=blue,domain=-3.06:8.1] plot(\x,{(-18.907457220923725--3.5422373506271194*\x)/-2.4749610560621527});
\draw [line width=.6pt,color=red,domain=-3.06:8.1] plot(\x,{(-0.3443569828272115--0.03772342484086216*\x)/-0.9294685796775379});
\draw [line width=.6pt,color=blue,domain=-3.06:8.1] plot(\x,{(-18.1496--4.9*\x)/6.54});
\draw [line width=.6pt,color=blue,domain=-3.06:8.1] plot(\x,{(-15.542--4.06*\x)/-7.48});
\draw [line width=.6pt,color=red,domain=-3.06:8.1] plot(\x,{(-0.5686031265763045--0.15701037509436988*\x)/0.5414882727956654});
\draw [line width=.6pt,color=red,domain=-3.06:8.1] plot(\x,{(--0.4717864255201903-0.14407743498201153*\x)/-1.7766544439534826});
\begin{small}
\draw [fill=white] (-1.22,2.74) circle (1.8pt);
\draw[color=black] (-1.49,2.45) node {$A_1$};
\draw [fill=white] (6.32,1.96) circle (1.8pt);
\draw[color=black] (6.47,1.6) node {$A_4$};
\draw [fill=white] (-0.22,-2.94) circle (1.8pt);
\draw[color=black] (-0.59,-2.64) node {$B_1$};
\draw [fill=white] (6.26,-1.32) circle (1.8pt);
\draw[color=black] (6.1,-1.7) node {$B_4$};
\draw [fill=white] (1.348718174934983,2.474270533627416) circle (1.8pt);
\draw[color=black] (0.85,2.2) node {$A_3$};
\draw [fill=white] (3.785038943937847,2.222237350627119) circle (1.8pt);
\draw[color=black] (3.8,1.75) node {$A_2$};
\draw [fill=white] (1.0725334508150293,-2.616866637296243) circle (1.8pt);
\draw[color=black] (1.58,-2.84) node {$B_3$};
\draw [fill=white] (3.98533199068538,-1.888667002328655) circle (1.8pt);
\draw[color=black] (3.98,-2.5) node {$B_2$};
\draw [fill=white] (0.35762780479019834,-0.9463766374039241) circle (1.8pt);
\draw [fill=white] (2.6859560264562155,0.2614757807323039) circle (1.8pt);
\draw[color=red] (-2.73,-2.9) node {$p_4$};
\draw [fill=white] (1.9794715219602823,-0.10502281220258479) circle (1.8pt);
\draw [fill=white] (5.227082818386968,0.15834157303293445) circle (1.8pt);
\draw [fill=white] (4.29761423870943,0.1960649978737966) circle (1.8pt);
\draw[color=red] (-2.73,0.75) node {$p_1$};
\draw [fill=white] (3.756125965913765,0.03905462277942673) circle (1.8pt);
\draw[color=red] (-2.73,-1.62) node {$p_2$};
\draw[color=red] (-2.73,-0.25) node {$p_3$};
\end{small}
\end{tikzpicture}
\caption{The generalized Pappus configuration: The Pappus lines $p_i$ coincide if and only if $\operatorname{CR}(A_1,A_2;A_3,A_4)=\operatorname{CR}(B_1,B_2;B_3,B_4)$.}\label{fig-pappus}
\end{center}
\end{figure}

A weaker form of this statement is Lemma 3.8 in \cite{hh}.

We get another interesting way of putting Theorem~\ref{thm:projective-quadrilateral}\,(\ref{diag}) by focussing on the two triangles $A_1A_4(g_1\times g_4)$ and $A_2A_3(g_2\times g_3)$:

\begin{corollary}\label{cor:triangles}
Let $A_1B_1C_1$ and $A_2B_2C_2$ be two triangles. Then the following statements are equivalent:
\begin{enumerate}
\item $A_1B_1C_1$ and $A_2B_2C_2$ are perspective from a point.\label{tr1}
\item $A_1B_1C_1$ and $A_2B_2C_2$ are perspective from a line.\label{tr2}
\item $\begin{displaystyle}
\operatorname{CR}(A_1,B_2;A_1',B_2')=\operatorname{CR}(A_2,B_1;A_2',B_1')
\end{displaystyle}$
where $A_1'=(A_1\times B_2)\times (B_1\times C_1)$, $B_1'=(A_2\times B_1)\times (A_1\times C_1)$, $A_2'=(A_2\times B_1)\times (B_2\times C_2)$, $B_2'=(A_1\times B_2)\times (A_2\times C_2)$ (see Figure~\ref{fig:desargues}).
\label{tr3}
\item $\begin{displaystyle}
\operatorname{CR}(B_1,C_2;B_1',C_2')=\operatorname{CR}(B_2,C_1;B_2',C_1')
\end{displaystyle}$
where $B_1'=(B_1\times C_2)\times (C_1\times A_1)$, $C_1'=(B_2\times C_1)\times (B_1\times A_1)$, $B_2'=(B_2\times C_1)\times (C_2\times A_2)$, $C_2'=(B_1\times C_2)\times (B_2\times A_2)$.
\label{tr4}
\item $\begin{displaystyle}
\operatorname{CR}(C_1,A_2;C_1',A_2')=\operatorname{CR}(C_2,A_1;C_2',A_1')
\end{displaystyle}$
where $C_1'=(C_1\times A_2)\times (A_1\times B_1)$, $A_1'=(C_2\times A_1)\times (C_1\times B_1)$, $C_2'=(C_2\times A_1)\times (A_2\times B_2)$, $A_2'=(C_1\times A_2)\times (C_2\times B_2)$.
\label{tr5}
\end{enumerate}
\end{corollary}
So Corollary~\ref{cor:triangles}  is Desargues' Theorem with an additional quantitative equivalence that tells us when the perspectives occur.

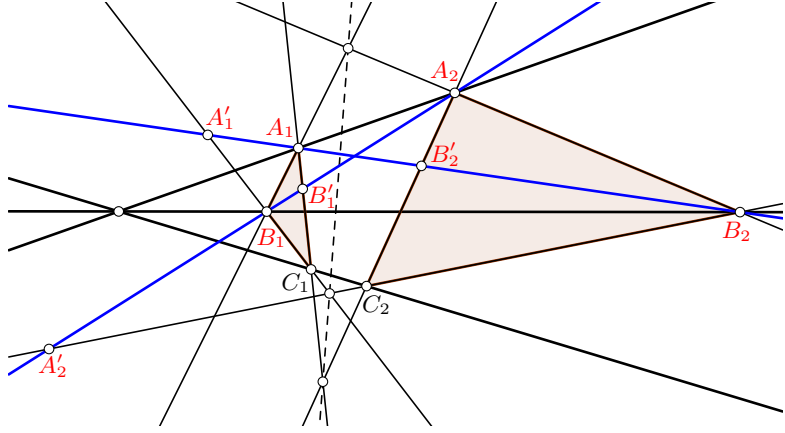
\begin{figure}[h]
\begin{center}
\definecolor{zzttqq}{rgb}{0.6,0.2,0.}
\begin{tikzpicture}[line cap=round,line join=round,,x=13,y=13]
\clip(-1.965276314819155,-4.83769698210485) rectangle (20.313369277346702,7.443775919728432);
\fill[line width=1pt,color=zzttqq,fill=zzttqq,fill opacity=0.10000000149011612] (6.37152802098225,3.219876948267131) -- (5.459171355371053,1.37532261197722) -- (6.737670460655945,-0.3021748405983691) -- cycle;
\fill[line width=1pt,color=zzttqq,fill=zzttqq,fill opacity=0.10000000149011612] (10.873865988440352,4.8216758374976845) -- (8.333479592182433,-0.786932825119719) -- (19.08458455180289,1.3603592918445258) -- cycle;
\draw [line width=1pt,domain=-1.965276314819155:20.313369277346702] plot(\x,{(--6.538848349918393--2.440865697682701*\x)/6.860787816766401});
\draw [line width=1pt,domain=-1.965276314819155:20.313369277346702] plot(\x,{(--12.295391280805697-0.009775230849952798*\x)/8.901203625931707});
\draw [line width=1pt,domain=-1.965276314819155:20.313369277346702] plot(\x,{(--15.7356-2.74*\x)/9.02});
\draw [line width=1pt,color=zzttqq] (6.37152802098225,3.219876948267131)-- (5.459171355371053,1.37532261197722);
\draw [line width=1pt,color=zzttqq] (5.459171355371053,1.37532261197722)-- (6.737670460655945,-0.3021748405983691);
\draw [line width=1pt,color=zzttqq] (6.737670460655945,-0.3021748405983691)-- (6.37152802098225,3.219876948267131);
\draw [line width=1pt,color=zzttqq] (10.873865988440352,4.8216758374976845)-- (8.333479592182433,-0.786932825119719);
\draw [line width=1pt,color=zzttqq] (8.333479592182433,-0.786932825119719)-- (19.08458455180289,1.3603592918445258);
\draw [line width=1pt,color=zzttqq] (19.08458455180289,1.3603592918445258)-- (10.873865988440352,4.8216758374976845);
\draw [line width=.6pt,domain=-1.965276314819155:20.313369277346702] plot(\x,{(--23.61978526539481-3.5220517888655*\x)/0.36614243967369475});
\draw [line width=.6pt,domain=-1.965276314819155:20.313369277346702] plot(\x,{(-10.916094770699523--1.677497452575589*\x)/-1.278499105284892});
\draw [line width=.6pt,domain=-1.965276314819155:20.313369277346702] plot(\x,{(-8.814953443696128--1.8445543362899108*\x)/0.9123566656111972});
\draw [line width=.6pt,domain=-1.965276314819155:20.313369277346702] plot(\x,{(--48.73833927416268-5.608608662617403*\x)/-2.5403863962579187});
\draw [line width=.6pt,domain=-1.965276314819155:20.313369277346702] plot(\x,{(-26.354812434208498--2.1472921169642447*\x)/10.751104959620458});
\draw [line width=.6pt,domain=-1.965276314819155:20.313369277346702] plot(\x,{(-77.22731556646258--3.4613165456531587*\x)/-8.210718563362539});
\draw [line width=.6pt,dash pattern=on 3pt off 3pt,domain=-1.965276314819155:20.313369277346702] plot(\x,{(-18.81929707206329--2.562447707569367*\x)/0.19259416893938575});
\draw [line width=1pt,color=blue,domain=-1.965276314819155:20.313369277346702] plot(\x,{(--52.78244651901416-1.859517656422605*\x)/12.71305653082064});
\draw [line width=1pt,color=blue,domain=-1.965276314819155:20.313369277346702] plot(\x,{(--11.367280843440051-3.4463532255204647*\x)/-5.414694633069299});
\begin{small}
\draw [fill=white] (1.2,1.38) circle (1.8pt);
\draw [fill=white] (6.37152802098225,3.219876948267131) circle (1.8pt);
\draw[color=red] (5.85,3.8) node {$A_1$};
\draw [fill=white] (5.459171355371053,1.37532261197722) circle (1.8pt);
\draw[color=red] (5.6,0.6177494053615746) node {$B_1$};
\draw [fill=white] (6.737670460655945,-0.3021748405983691) circle (1.8pt);
\draw[color=black] (6.3,-0.7) node {$C_1$};
\draw [fill=white] (10.873865988440352,4.8216758374976845) circle (1.8pt);
\draw[color=red] (10.59,5.45) node {$A_2$};
\draw [fill=white] (8.333479592182433,-0.786932825119719) circle (1.8pt);
\draw[color=black] (8.62,-1.33) node {$C_2$};
\draw [fill=white] (19.08458455180289,1.3603592918445258) circle (1.8pt);
\draw[color=red] (18.99,0.8) node {$B_2$};
\draw [fill=white] (7.076548205504936,-3.5619580236562456) circle (1.8pt);
\draw [fill=white] (7.269142374444321,-0.9995103160868789) circle (1.8pt);
\draw [fill=white] (7.803933674768116,6.115838757508587) circle (1.8pt);
\draw [fill=white] (3.7625015354298275,3.601494917494062) circle (1.8pt);
\draw[color=red] (4.121711895826884,4.1) node {$A_1'$};
\draw [fill=white] (6.494760965201638,2.034456290607596) circle (1.8pt);
\draw[color=red] (7.06,1.8270848114501907) node {$B_1'$};
\draw [fill=white] (-0.8059898376906127,-2.6123368867218586) circle (1.8pt);
\draw[color=red] (-0.6887556083922814,-3.0908791733101824) node {$A_2'$};
\draw [fill=white] (9.913668838417243,2.7017738960556272) circle (1.8pt);
\draw[color=red] (10.59,3.04) node {$B_2'$};
\end{small}
\end{tikzpicture}
\caption{Quantitative version of Desargue's Theorem in Corollary~\ref{cor:triangles}(\ref{tr3}).}\label{fig:desargues}
\end{center}
\end{figure}

Next, we want to demonstrate how we can use Theorem~\ref{thm:projective-quadrilateral}\,(\ref{diag}) to expand the special case $n=4$ of the generalized Ceva's Theorem for $n$-gons which can be found in~\cite[Theorem 2]{gruenbaum95}. For $n=4$, the theorem reads as follows:
\begin{theorem}[Ceva's Theorem for quadrilaterals]\label{thm:ceva4}
Let $A_1,A_2,A_3,A_4$ be the vertices of a quadrilateral and $P$ a given point. For $i\in\{1,2,3,4\}$ let $D_i=(A_i\times P)\times(A_{i-1}\times A_{i+1})$. Then,
\begin{equation}\label{eq:ceva4}
\frac{A_1D_2}{D_2 A_3}\cdot \frac{A_2D_3}{D_3 A_4}\cdot \frac{A_3D_4}{D_4 A_1}\cdot \frac{A_4D_1}{D_1 A_2}=1.
\end{equation}
\end{theorem}
Note that the implication is only one-sided. But now Theorem~\ref{thm:projective-quadrilateral}\,(\ref{diag}) tells us how we can generalize Theorem~\ref{thm:ceva4} to get an equivalence:

\begin{theorem}[Two-sided Ceva's Theorem for quadrilaterals]\label{thm:ceva4full}
Let $A_1,A_2,A_3,A_4$ be the vertices of a quadrilateral and let $g_1,g_2,g_3,g_4$ be lines such that for each $i\in\{1,2,3,4\}$, $g_i$ passes through $A_i$. Then the following statements are equivalent:
\begin{enumerate}
\item $\begin{displaystyle}
\frac{A_1D_2}{D_2 A_3}\cdot \frac{A_2D_3}{D_3 A_4}\cdot \frac{A_3D_4}{D_4 A_1}\cdot \frac{A_4D_1}{D_1 A_2}=1
\end{displaystyle}$\\
where $D_i=g_i\times(A_{i-1}\times A_{i+1})$ for $i\in\{1,2,3,4\}$.\label{ceva4full1}
\item $g_1\times g_2, g_3\times g_4,(A_1\times A_4)\times(A_2\times A_3)$ are collinear (see Figure~\ref{fig:cevafull}).\label{ceva4full2}
\item $g_1\times g_4, g_2\times g_3,(A_1\times A_2)\times(A_3\times A_4)$ are collinear.\label{ceva4full3}
\end{enumerate}
\end{theorem}
It's easy to see that when $g_1,g_2,g_3,g_4$ are concurrent, (\ref{ceva4full2}) and (\ref{ceva4full3}) are trivially satisfied, hence Theorem~\ref{thm:ceva4full} is in fact a generalization of Theorem~\ref{thm:ceva4}.

\begin{figure}[h]
\begin{center}
\begin{tikzpicture}[line cap=round,line join=round,x=33,y=33]
\clip(-2.0099660178106946,-1.4754692113985637) rectangle (4.252041024060651,6.0674029072191935);
\draw [line width=1pt,domain=-2.0099660178106946:4.252041024060651] plot(\x,{(-2.5725420795613423--0.739466331521546*\x)/4.375759159365857});
\draw [line width=1pt,domain=-2.0099660178106946:4.252041024060651] plot(\x,{(-10.386817528055786--2.988685179074961*\x)/-2.1218370968324396});
\draw [line width=1pt,domain=-2.0099660178106946:4.252041024060651] plot(\x,{(-6.019588429653726--0.21994652833019224*\x)/-1.9148285995804946});
\draw [line width=1pt,domain=-2.0099660178106946:4.252041024060651] plot(\x,{(-3.3023590724488665-3.948098038926699*\x)/-0.33909346295292253});
\draw [line width=0.6pt,domain=-2.0099660178106946:4.252041024060651] plot(\x,{(--11.150276804804935-3.208631707405153*\x)/4.036665696412934});
\draw [line width=0.6pt,domain=-2.0099660178106946:4.252041024060651] plot(\x,{(-1.6874340332621276-3.728151510596507*\x)/-2.253922062533417});
\draw [line width=0.6pt,color=red,domain=-2.0099660178106946:4.252041024060651] plot(\x,{(-0.5170959522942737--1.8620672819316786*\x)/2.963454737882141});
\draw [line width=0.6pt,color=red,domain=-2.0099660178106946:4.252041024060651] plot(\x,{(-2.0041715595664136-0.5653561102259026*\x)/-0.9268672155905692});
\draw [line width=0.6pt,color=red,domain=-2.0099660178106946:4.252041024060651] plot(\x,{(--5.151837863456965-0.5669883437295753*\x)/1.7050181339577692});
\draw [line width=0.6pt,color=red,domain=-2.0099660178106946:4.252041024060651] plot(\x,{(-5.576303312268419--1.6047953586402863*\x)/-2.958492954172865});
\draw [line width=0.6pt,dash pattern=on 3pt off 3pt,color=blue,domain=-2.0099660178106946:4.252041024060651] plot(\x,{(--5.948254174473577-2.6995201671464026*\x)/1.2827929999477856});
\begin{small}
\draw [fill=white] (-0.9,-0.74) circle (1.8pt);
\draw[color=black] (-.67,-0.9) node {$A_1$};
\draw [fill=white] (3.4757591593658566,-5.336684784540555E-4) circle (1.8pt);
\draw[color=black] (3.339909833044268,-0.2) node {$A_2$};
\draw [fill=white] (1.353922062533417,2.9881515105965066) circle (1.8pt);
\draw[color=black] (1.8,2.77) node {$A_3$};
\draw [fill=white] (-0.5609065370470775,3.208098038926699) circle (1.8pt);
\draw[color=black] (-0.7873220809163923,3.07) node {$A_4$};
\draw [fill=white] (2.0634547378821413,1.1220672819316786) circle (1.8pt);
\draw[color=red] (2.1,1.39) node {$D_1$};
\draw[color=red] (-1.8094265360978723,-1.14) node {$g_1$};
\draw [fill=white] (0.4270548469428478,2.422795400370604) circle (1.8pt);
\draw[color=red] (0.41,2.18) node {$D_3$};
\draw[color=red] (-1.8094265360978723,0.85) node {$g_3$};
\draw [fill=white] (1.1441115969106919,2.6411096951971236) circle (1.8pt);
\draw[color=red] (1.33,2.4) node {$D_4$};
\draw[color=red] (-1.8094265360978723,3.8) node {$g_4$};
\draw [fill=white] (0.5172662051929919,1.6042616901618323) circle (1.8pt);
\draw[color=red] (0.21,1.57) node {$D_2$};
\draw[color=red] (-1.8094265360978723,2.66) node {$g_2$};
\draw [fill=white] (-0.3711096341262466,5.417921766257375) circle (1.8pt);
\draw[color=blue] (1.15,5.45) node {$(A_1\times A_4)\times(A_2\times A_3)$};
\draw [fill=white] (1.7589440302783559,0.9307299828062336) circle (1.8pt);
\draw[color=blue] (1.15,0.93) node {$g_1\times g_2$};
\draw [fill=white] (0.911683365821539,2.718401599110972) circle (1.8pt);
\draw [color=white,fill=white] (0.,2.718401599110972) circle (5pt);

\draw[color=blue] (0.3,2.72) node {$g_3\times g_4$};
\end{small}
\end{tikzpicture}
\caption{Collinear points in Theorem~\ref{thm:ceva4full}(\ref{ceva4full2}).}\label{fig:cevafull}
\end{center}
\end{figure}
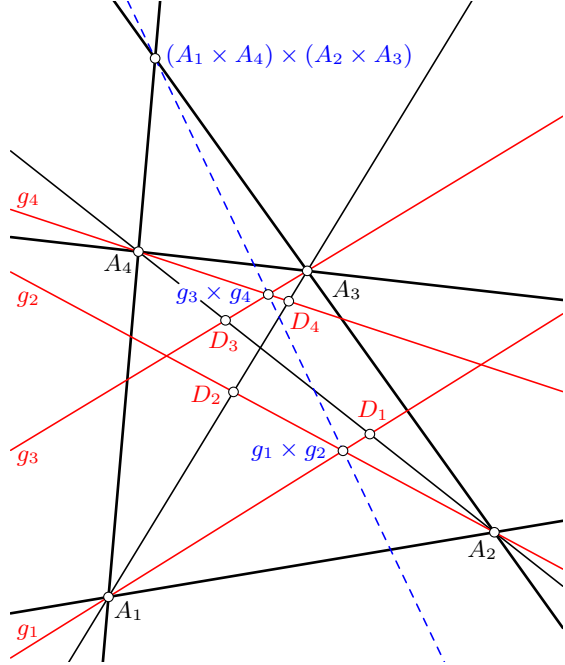

Can the cases $n>4$ also be generalized in a similar fashion? The answer is yes:

\begin{theorem}[Two-sided Ceva's Theorem for $n$-gons]\label{thm:cevafull}
For $n\geq 3$, let $P^n=[A_1^n,\dots,A_n^n]$ be an $n$-gon and let $g_1^n\dots,g_n^n$ 
be lines such that for each $i\in\{1,\dots,n\}$, $g_i^n$ passes through $A_i^n$.
For $m\in\{4,\dots,n\}$, the $m$-gon $P^m$ can be reduced in a step to the $(m-1)$-gon $P^{m-1}$ by obeying the following procedure:
\begin{itemize}
\item Choose two adjacent vertices $A_i^m$ and $A_{i+1}^m$, $i\in\{1,\dots,m\}$.
\item Replace $A_i^m$ and $A_{i+1}^m$ by\\
\[A_i^{m-1}=(A_{i-1}^m\times A_i^m)\times(A_{i+1}^m\times A_{i+2}^m).\]
\item Replace $g_i^m$ and $g_{i+1}^m$ by\\
\[g_i^{m-1}=A_i^{m-1}\times (g_i^m\times g_{i+1}^m).\]
\item For the other vertices, let
\[A_j^{m-1}=\begin{cases}A_j^m&\text{for $j\in\{1,\dots,i-1\}$,}\\A_{j+1}^m&\text{for $j\in\{i+1,\dots,m-1\}.$}\end{cases}\]
\item Similarly, for the other lines, let
\[g_j^{m-1}=\begin{cases}g_j^m&\text{for $j\in\{1,\dots,i-1\}$,}\\g_{j+1}^m&\text{for $j\in\{i+1,\dots,m-1\}.$}\end{cases}\]
\end{itemize}
If it is possible to reduce $P^n$ in $(n-3)$ steps to a triangle $P^3$ for which $g_1^3,g_2^3,g_3^3$ are concurrent, then every sequence of $(n-3)$ steps will lead to such a triangle (see Figure~\ref{fig-ceva-n}). Furthermore, this is the case if and only if
\begin{equation}\label{eq:cevafull}\frac{A_1^nD_2^n}{D_2^n A_3^n}\cdot \frac{A_2^nD_3^n}{D_3^n A_4^n}\cdot \ldots\cdot \frac{A_n^nD_1^n}{D_1^n A_{2}^n}=1,\end{equation}
where $D_i^n=g_i^n\times (A_{i-1}^n\times A_{i+1}^n)$ for $i\in\{1,\dots,n\}$.
\end{theorem}
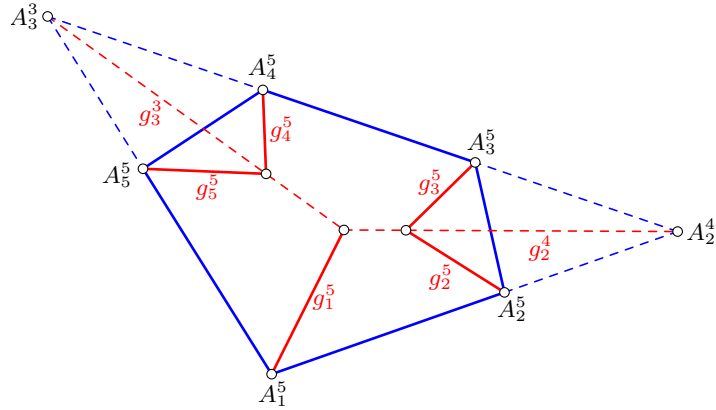
\begin{figure}[h]
\begin{center}
\begin{tikzpicture}[line cap=round,line join=round,x=20,y=20]
\clip(-.2,-2.77029987099037) rectangle (12.99,5.078267726749624);
\draw [line width=1pt,color=blue] (2.2668826205398047,1.8477691952478335)-- (4.509291387903656,3.33502495254647);
\draw [line width=1pt,color=blue] (8.4822480090719,1.967285787881993)-- (9.029267073696232,-0.4842962631811971);
\draw [line width=.6pt,dash pattern=on 3pt off 3pt,color=red] (0.4818689214134097,4.721514654125079)-- (6.026406258453216,0.693177875739591);
\draw [line width=1pt,color=red] (4.673717524751815,-2.027064303022566)-- (6.026406258453216,0.693177875739591);
\draw [line width=.6pt,dash pattern=on 3pt off 3pt,color=red] (6.026406258453216,0.693177875739591)-- (12.269585029382757,0.6634484530208788);
\draw [line width=1pt,color=red] (4.509291387903656,3.33502495254647)-- (4.569338258633309,1.7517983260108902);
\draw [line width=1pt,color=red] (2.2668826205398047,1.8477691952478335)-- (4.569338258633309,1.7517983260108902);
\draw [line width=1pt,color=red] (8.4822480090719,1.967285787881993)-- (7.185898236487691,0.6876564853679983);
\draw [line width=1pt,color=red] (7.185898236487691,0.6876564853679983)-- (9.029267073696232,-0.4842962631811971);
\draw [line width=.6pt,dash pattern=on 3pt off 3pt,color=blue] (0.4818689214134097,4.721514654125079)-- (2.2668826205398047,1.8477691952478335);
\draw [line width=1pt,color=blue] (2.2668826205398047,1.8477691952478335)-- (4.673717524751815,-2.027064303022566);
\draw [line width=1pt,color=blue] (4.673717524751815,-2.027064303022566)-- (9.029267073696232,-0.4842962631811971);
\draw [line width=.6pt,dash pattern=on 3pt off 3pt,color=blue] (9.029267073696232,-0.4842962631811971)-- (12.269585029382757,0.6634484530208788);
\draw [line width=.6pt,dash pattern=on 3pt off 3pt,color=blue] (12.269585029382757,0.6634484530208788)-- (8.4822480090719,1.967285787881993);
\draw [line width=1pt,color=blue] (4.509291387903656,3.33502495254647)-- (8.4822480090719,1.967285787881993);
\draw [line width=.6pt,dash pattern=on 3pt off 3pt,color=blue] (0.4818689214134097,4.721514654125079)-- (4.509291387903656,3.33502495254647);
\begin{small}
\draw [fill=white] (4.673717524751815,-2.027064303022566) circle (1.8pt);
\draw[color=black] (4.7,-2.4) node {$A_1^5$};
\draw [fill=white] (12.269585029382757,0.6634484530208788) circle (1.8pt);
\draw (12.269585029382757,0.6634484530208788) node[anchor=west] {$A_2^4$};
\draw [color=red] (9.7,0.75) node[anchor=north] {$g_2^4$};
\draw [fill=white] (0.4818689214134097,4.721514654125079) circle (1.8pt);
\draw (0.4818689214134097,4.721514654125079) node[anchor=east] {$A_3^3$};

\draw[color=red] (2.8,2.9) node[anchor=east] {$g_3^3$};

\draw [fill=white] (2.2668826205398047,1.8477691952478335) circle (1.8pt);
\draw[color=black] (1.78,1.733707670894513) node {$A_5^5$};
\draw [fill=white] (4.509291387903656,3.33502495254647) circle (1.8pt);
\draw[color=black] (4.502773344119215,3.7701731271262915) node {$A_4^5$};
\draw [fill=white] (8.4822480090719,1.967285787881993) circle (1.8pt);
\draw[color=black] (8.60543367930148,2.3580255479874666) node {$A_3^5$};
\draw [fill=white] (9.029267073696232,-0.4842962631811971) circle (1.8pt);
\draw[color=black] (9.17029271095701,-.8) node {$A_2^5$};
\draw [fill=white] (6.026406258453216,0.693177875739591) circle (1.8pt);
\draw[color=red] (5.66,-0.5703225898056747) node {$g_1^5$};
\draw [fill=white] (4.569338258633309,1.7517983260108902) circle (1.8pt);
\draw[color=red] (4.86,2.5958609297371633) node {$g_4^5$};
\draw[color=red] (3.4771082603236487,1.5) node {$g_5^5$};
\draw [fill=white] (7.185898236487691,0.6876564853679983) circle (1.8pt);
\draw[color=red] (7.64,1.6) node {$g_3^5$};
\draw[color=red] (7.84,-.2) node {$g_2^5$};
\end{small}
\end{tikzpicture}
\caption{Ceva's Theorem for a pentagon $A_1^5\ldots A_5^5$.}\label{fig-ceva-n}
\end{center}
\end{figure}

In the proof we will make use of the area principle that can also be found in ~\cite{gruenbaum95}:

\begin{lemma}[Area principle, see Figure~\ref{fig-ABC}] Whenever $D$ is the intersection point of two lines $A_1\times A_2$ and $B\times C$, then there holds
\begin{equation}\label{eq:areapr}\frac{A_1D}{DA_2}=\frac{[A_1BC]}{[BA_2C]}.\end{equation}
%where $[A_1BC],[BA_2C]$ denote the areas of the triangles $A_1BC,BA_2C$ and the ratio $[A_1BC]/[BA_2C]$ is negative when these triangles have reversed orientation.
\end{lemma}
\begin{proof}
Each side of~(\ref{eq:areapr}) is equal to the ratio of the two heights of the triangles $A_1BC$ and $BA_2C$ with respect to the base $BC$.
\end{proof}

At some point, we will also need a slightly more refined version of the area principle:

\begin{lemma}[Area principle for unequal bases, see Figure~\ref{fig-ABC}] Whenever $D$ is the intersection point of two lines $A_1\times A_2$ and $B\times C_1$, on which also the point $C_2$ lies, then there holds
\begin{equation}\label{eq:areapr2}\frac{A_1D}{DA_2}=\frac{[A_1BC_1]}{[BA_2C_2]}\cdot \frac{BC_2}{BC_1}.\end{equation}
\end{lemma}
{\em Proof.} As above, we have
\begin{equation}
\frac{A_1D}{DA_2}=\frac{[A_1BC_2]}{[BA_2C_2]}=\frac{[A_1BC_1]\cdot BC_2/BC_1}{[BA_2C_2]}=\frac{[A_1BC_1]}{[BA_2C_2]}\cdot \frac{BC_2}{BC_1}.\tag*{$\Box$}
\end{equation}
\begin{figure}[h]
\begin{center}
\definecolor{gray1}{rgb}{0.25098039215686274,0.25098039215686274,0.25098039215686274}
\definecolor{gray2}{rgb}{0.5019607843137255,0.5019607843137255,0.5019607843137255}
\begin{tikzpicture}[line cap=round,line join=round,x=15,y=15]

\clip(-4.35,-0.6321037166824842) rectangle (6.411046961169184,4.683808701461465);
\fill[line width=1.pt,color=gray2,fill=gray2,fill opacity=0.10000000149011612] (-3.66,3.98) -- (-1.3843022191687888,0.02246758853673282) -- (5.6176271972671,1.4902941517555843) -- cycle;

\fill[line width=1.pt,color=gray1,fill=gray1,fill opacity=0.3100000023841858] (-3.66,3.98) -- (5.04239786843809,3.6721985024863097) -- (5.6176271972671,1.4902941517555843) -- cycle;

\draw [shift={(-0.5464938813578679,3.1444723130919185)},line width=0.6pt] (0,0) -- (164.9782614177078:0.5950648229265628) arc (164.9782614177078:254.9782614177078:0.5950648229265628) ;

\draw [shift={(4.534847264951693,1.7808645177044837)},line width=0.6pt] (0,0) -- (-15.021738582292203:0.5950648229265628) arc (-15.021738582292203:74.9782614177078:0.5950648229265628) ;

\draw [line width=1.pt] (-1.3843022191687888,0.02246758853673282)-- (5.04239786843809,3.6721985024863097);
\draw [line width=1.pt] (-3.66,3.98)-- (5.04239786843809,3.6721985024863097);
\draw [line width=1.pt] (5.04239786843809,3.6721985024863097)-- (5.6176271972671,1.4902941517555843);
\draw [line width=1.pt] (5.6176271972671,1.4902941517555843)-- (-3.66,3.98);
\draw [line width=1.pt] (-3.66,3.98)-- (-1.3843022191687888,0.02246758853673282);
\draw [line width=1.pt] (-1.3843022191687888,0.02246758853673282)-- (5.6176271972671,1.4902941517555843);
\draw [line width=.6pt,dash pattern=on 3pt off 3pt] (-1.3843022191687888,0.02246758853673282)-- (-0.5464938813578679,3.1444723130919185);
\draw [line width=.6pt,dash pattern=on 3pt off 3pt] (5.04239786843809,3.6721985024863097)-- (4.534847264951693,1.7808645177044837);
\fill[line width=1pt] (-0.8497021773469153,2.969568275143413) circle (0.04);
\fill[line width=1pt] (4.83805556094074,1.955768555652989) circle (0.04);

\draw[->,line width=.6pt] (.6,2) arc (20:340:.4) ;
\draw[->,line width=.6pt] (2.5,3.26) arc (20:340:.4) ;
\begin{small}
\draw [fill=white] (-1.3843022191687888,0.02246758853673282) circle (1.8pt);
\draw[color=black] (-1.7710943540710544,-0.27506482292654677) node {$A_1$};
\draw [fill=white] (5.04239786843809,3.6721985024863097) circle (1.8pt);
\draw[color=black] (5.6,3.8705534434618314) node {$A_2$};
\draw [fill=white] (5.6176271972671,1.4902941517555843) circle (1.8pt);
\draw[color=black] (6,1.5) node {$B$};
\draw [fill=white] (-3.66,3.98) circle (1.8pt);
\draw[color=black] (-4.13,4.078826131486128) node {$C$};
\draw [fill=white] (2.6178586560658488,2.29529969052183) circle (1.8pt);
\draw[color=black] (2.602632094439182,1.8) node {$D$};
\end{small}
\end{tikzpicture}\quad
\begin{tikzpicture}[line cap=round,line join=round,x=15,y=15]
\clip(-4.54,-0.6321037166824842) rectangle (6.411046961169184,4.683808701461465);

\fill[line width=1.pt,color=gray2,fill=gray2,fill opacity=0.10000000149011612] (-3.66,3.98) -- (-1.3843022191687888,0.02246758853673282) -- (5.6176271972671,1.4902941517555843) -- cycle;

\fill[line width=1.pt,color=gray1,fill=gray1,fill opacity=0.3100000023841858] (5.6176271972671,1.4902941517555843) -- (5.04239786843809,3.6721985024863097) -- (-1.8951565749769639,3.5063938824994096) -- cycle;

\draw [shift={(-0.5464938813578679,3.1444723130919185)},line width=0.6pt] (0,0) -- (164.9782614177078:0.5950648229265628) arc (164.9782614177078:254.9782614177078:0.5950648229265628) ;

\draw [shift={(4.534847264951693,1.7808645177044837)},line width=0.6pt] (0,0) -- (-15.021738582292203:0.5950648229265628) arc (-15.021738582292203:74.9782614177078:0.5950648229265628) ;

\draw [line width=1.pt] (-1.3843022191687888,0.02246758853673282)-- (5.04239786843809,3.6721985024863097);
\draw [line width=1.pt] (5.04239786843809,3.6721985024863097)-- (5.6176271972671,1.4902941517555843);
\draw [line width=1.pt] (5.6176271972671,1.4902941517555843)-- (-3.66,3.98);
\draw [line width=1.pt] (-3.66,3.98)-- (-1.3843022191687888,0.02246758853673282);
\draw [line width=1.pt] (-1.3843022191687888,0.02246758853673282)-- (5.6176271972671,1.4902941517555843);
\draw [line width=.6pt,dash pattern=on 3pt off 3pt] (-1.3843022191687888,0.02246758853673282)-- (-0.5464938813578679,3.1444723130919185);
\draw [line width=.6pt,dash pattern=on 3pt off 3pt] (5.04239786843809,3.6721985024863097)-- (4.534847264951693,1.7808645177044837);
\draw [line width=1pt] (5.04239786843809,3.6721985024863097)-- (-1.8951565749769639,3.5063938824994096);
\fill[line width=1pt] (-0.8497021773469153,2.969568275143413) circle (0.04);
\fill[line width=1pt] (4.83805556094074,1.955768555652989) circle (0.04);

\draw[->,line width=.6pt] (.6,2) arc (20:340:.4) ;
\draw[->,line width=.6pt] (2.7,3.1) arc (20:340:.4) ;
\begin{small}
\draw [fill=white] (-1.3843022191687888,0.02246758853673282) circle (1.8pt);
\draw[color=black] (-1.7710943540710544,-0.27506482292654677) node {$A_1$};
\draw [fill=white] (5.04239786843809,3.6721985024863097) circle (1.8pt);
\draw[color=black] (5.6,3.8705534434618314) node {$A_2$};
\draw [fill=white] (5.6176271972671,1.4902941517555843) circle (1.8pt);
\draw[color=black] (6,1.5) node {$B$};
\draw [fill=white] (-3.66,3.98) circle (1.8pt);
\draw[color=black] (-4.13,4.078826131486128) node {$C_1$};
\draw [fill=white] (2.6178586560658488,2.29529969052183) circle (1.8pt);
\draw[color=black] (2.602632094439182,1.8) node {$D$};
\draw [fill=white] (-1.8951565749769639,3.5063938824994096) circle (1.8pt);
\draw[color=black] (-2.1479687419245446,3.037462691364645) node {$C_2$};
\end{small}
\end{tikzpicture}

\caption{Area principles.}\label{fig-ABC}
\end{center}
\end{figure}
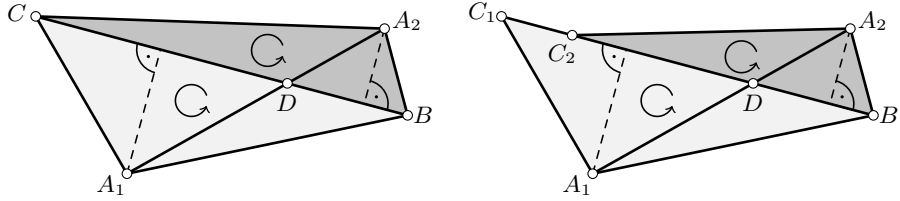

\begin{proof}[Proof of Theorem~\ref{thm:cevafull}] %We are now ready to prove Theorem~\ref{thm:cevafull}. 
We want to show that in every step, the left-hand side of~(\ref{eq:cevafull}) is constant, i.e.:
\begin{equation}\label{eq:cevafull2}
\frac{A_1^mD_2^m}{D_2^m A_3^m}\cdot \frac{A_2^mD_3^m}{D_3^m A_4^m}\cdot \ldots\cdot \frac{A_m^mD_1^m}{D_1^m A_{2}^m}
=\frac{A_1^{m-1}D_2^{m-1}}{D_2^{m-1} A_3^{m-1}}\cdot \frac{A_2^{m-1}D_3^{m-1}}{D_3^{m-1} A_4^{m-1}}\cdot \ldots\cdot \frac{A_{m-1}^{m-1}D_1^{m-1}}{D_1^{m-1} A_{2}^{m-1}}
\end{equation}
By symmetry we can assume that in this step, $A_{m-1}^m$ and $A_{m}^m$ are replaced by $A_{m-1}^{m-1}$ (this will make the notation easier since $A_j^{m-1}=A_j^m$ for all $j\in\{1,\dots,m-2\}$). Having a closer look at~(\ref{eq:cevafull2}), we see that it is enough to show
\begin{multline}\label{eq:cevafull3}
\frac{A_{m-3}^mD_{m-2}^m}{D_{m-2}^m A_{m-1}^m}\cdot \frac{A_{m-2}^mD_{m-1}^m}{D_{m-1}^m A_{m}^m}\cdot \frac{A_{m-1}^mD_{m}^m}{D_{m}^m A_{1}^m}\cdot \frac{A_{m}^mD_{1}^m}{D_{1}^m A_{2}^m}\\
=\frac{A_{m-3}^{m-1}D_{m-2}^{m-1}}{D_{m-2}^{m-1} A_{m-1}^{m-1}}\cdot \frac{A_{m-2}^{m-1}D_{m-1}^{m-1}}{D_{m-1}^{m-1} A_{1}^{m-1}}\cdot\frac{A_{m-1}^{m-1}D_{1}^{m-1}}{D_{1}^{m-1} A_{2}^{m-1}}
\end{multline}
Let $G_{m-2}$ and $G_{1}$ be the intersections of $g_{m-2}^m$ and $g_1^m$ with the new line $g_{m-1}^{m-1}$. Further, let $G_{m-1}=g_{m-1}^m\times g_m^m$. By definition, $G_{m-1}$ also lies on $g_{m-1}^{m-1}$ (see Figure~\ref{fig-ceva-n-proof}).

\begin{figure}[h]
\begin{center}
\begin{tikzpicture}[line cap=round,line join=round,x=22,y=22]
\clip(-2.03945284166284,-3.5419349255492967) rectangle (14.934475496490352,5.964272265021602);
\draw [line width=1.pt,color=blue] (0.03592758063272827,-1.0905874873831358)-- (6.30883577428098,-2.7851645823497253);
\draw [line width=1.pt,color=blue] (0.6602454577256823,4.394491004219247)-- (5.060200020095072,5.286373685780609);
\draw [line width=1.pt,color=blue] (6.30883577428098,-2.7851645823497253)-- (13.369573669975104,2.38775497070618);
\draw [line width=1.pt,color=blue] (5.060200020095072,5.286373685780609)-- (13.369573669975104,2.38775497070618);
\draw [line width=1.pt,color=blue] (8.660147857950152,4.030577928389302)-- (9.187938202710452,-0.6758432242582386);
\draw [line width=.6pt,color=red] (0.03592758063272827,-1.0905874873831358)-- (9.187938202710452,-0.6758432242582386);
\draw [line width=.6pt,color=red] (6.30883577428098,-2.7851645823497253)-- (8.660147857950152,4.030577928389302);
\draw [line width=.6pt,color=red] (9.187938202710452,-0.6758432242582386)-- (5.060200020095072,5.286373685780609);
\draw [line width=.6pt,color=red] (0.6602454577256823,4.394491004219247)-- (8.660147857950152,4.030577928389302);
\draw [line width=.6pt,color=orange] (5.060200020095072,5.286373685780609)-- (6.30883577428098,-2.7851645823497253);
\draw [line width=.6pt,color=orange] (0.6602454577256823,4.394491004219247)-- (13.369573669975104,2.38775497070618);
\draw [line width=.6pt,color=orange] (0.03592758063272827,-1.0905874873831358)-- (13.369573669975104,2.38775497070618);
%\draw [line width=0.6pt,domain=-2.03945284166284:6.30883577428098] plot(\x,{(-2.7970026381870445--1.8359014505923033*\x)/-3.15435510240194});
%\draw [line width=0.6pt,domain=-2.03945284166284:5.060200020095072] plot(\x,{(-2.7595338546296198-1.0237132260280122*\x)/-1.5019232489947867});
%\draw [line width=0.6pt,domain=-2.03945284166284:8.660147857950152] plot(\x,{(-0.532030306938232-0.9146806516221422*\x)/-2.09729228487347});
%\draw [line width=0.6pt,domain=-2.03945284166284:9.187938202710452] plot(\x,{(-3.9591833407094676--0.5745839336589982*\x)/-1.9532019951350392});
%\draw [line width=0.6pt,domain=-2.03945284166284:13.369573669975104] plot(\x,{(-11.752079592496283-1.0750628842814438*\x)/-10.941328714322147});
\draw [line width=.6pt,color=black] (-1.3083754597218915,0.9455427185459117)-- (13.369573669975104,2.38775497070618);
\draw [line width=.6pt,color=black] (-0.2754585430328719,1.0470341067193807)-- (6.30883577428098,-2.7851645823497253);
\draw [line width=.6pt,color=black] (-1.3083754597218915,0.9455427185459117)-- (5.060200020095072,5.286373685780609);
\draw [line width=.6pt,color=black] (2.4282449556529566,1.3126920864247362)-- (9.187938202710452,-0.6758432242582386);
\draw [line width=.6pt,color=black] (2.4282449556529566,1.3126920864247362)-- (8.660147857950152,4.030577928389302);
\begin{small}
\draw [fill=white] (0.03592758063272827,-1.0905874873831358) circle (1.8pt);
\draw[color=black] (0.03592758063272827,-1.12) node[anchor=north] {$A_{m-3}^m$};
\draw [fill=white] (6.30883577428098,-2.7851645823497253) circle (1.8pt);
\draw[color=black] (6.30883577428098,-2.81) node[anchor=north] {$A_{m-2}^m$};
\draw [fill=white] (5.060200020095072,5.286373685780609) circle (1.8pt);
\draw[color=black] (5.060200020095072,5.286373685780609) node[anchor=south] {$A_1^m$};
\draw [fill=white] (0.6602454577256823,4.394491004219247) circle (1.8pt);
\draw[color=black] (0.6602454577256823,4.394491004219247) node[anchor=south] {$A_2^m$};
\draw [fill=white] (13.369573669975104,2.38775497070618) circle (1.8pt);
\draw[color=black] (13.1,2.3) node[anchor=south west] {$A_{m-1}^{m-1}$};
\draw [fill=white] (9.187938202710452,-0.6758432242582386) circle (1.8pt);
\draw[color=black] (9.43,-0.6758432242582386) node[anchor=north] {$A_{m-1}^m$};
\draw [fill=white] (8.660147857950152,4.030577928389302) circle (1.8pt);
\draw[color=black] (8.7,4.030577928389302) node[anchor=south] {$A_m^m$};
\draw [fill=red] (3.15448067187904,-0.949263131757422) circle (1.8pt);
\draw[color=red] (3.5,-0.67) node {$D_{m-2}^m$};
\draw [fill=red] (7.234736207575413,-0.10125929059924044) circle (1.8pt);
\draw[color=red] (6.56,-0.25) node {$D_{m-1}^m$};
\draw [fill=red] (6.562855573076682,3.11589727676716) circle (1.8pt);
\draw[color=red] (6.4,3.0) node[anchor=north] {$D_m^m$};
\draw [fill=red] (3.5582767711002856,4.262660459752597) circle (1.8pt);
\draw[color=red] (3.45,4.26) node[anchor=south] {$D_1^m$};
\draw [fill=white] (2.4282449556529566,1.3126920864247362) circle (1.8pt);
\draw[color=black] (2.66,.96) node {$G_{m-1}$};
\draw [fill=orange] (2.356973706295705,-0.4850971937319243) circle (1.8pt);
\draw[color=orange] (2.5,-.05) node {$D_{m-2}^{m-1}$};
\draw [fill=orange] (5.626302753969001,1.6269238703812807) circle (1.8pt);
\draw[color=orange] (5.6,1.7) node[anchor=north west] {$D_{m-1}^{m-1}$};
\draw [fill=orange] (3.1702424828668443,3.998175684460116) circle (1.8pt);
\draw[color=orange] (3.37,3.97) node[anchor=north] {$D_1^{m-1}$};
\draw [fill=white] (-0.2754585430328719,1.0470341067193807) circle (1.8pt);
\draw[color=black] (.1,1.33) node {$G_{m-2}$};
\draw [fill=white] (-1.3083754597218915,0.9455427185459117) circle (1.8pt);
\draw[color=black] (-1.5,1.25) node {$G_1$};
\end{small}
\end{tikzpicture}
\caption{Proof of Theorem~\ref{thm:cevafull}.}\label{fig-ceva-n-proof}
\end{center}
\end{figure}
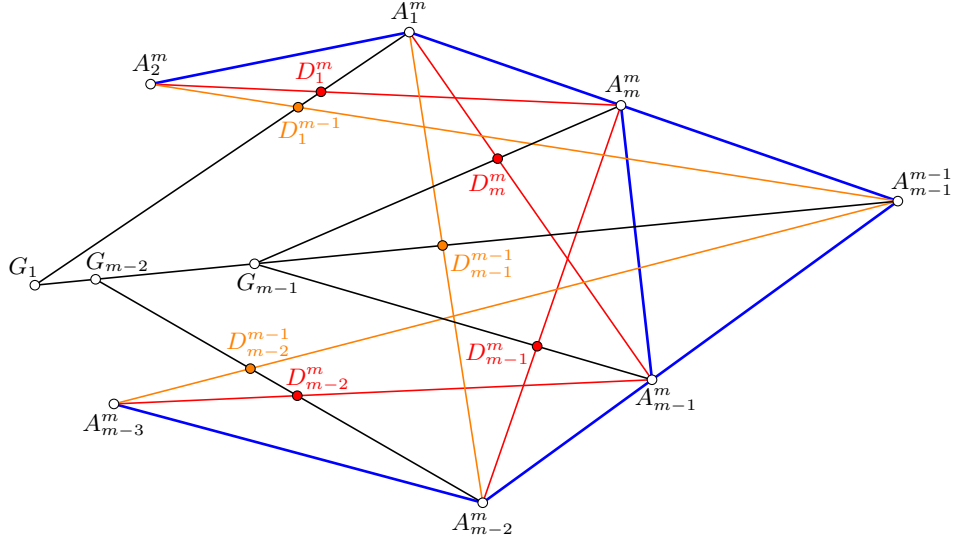

By using the area principle, we can rewrite all the fractions in~(\ref{eq:cevafull3}). Let us start with the left-hand side:
\begin{align}
\frac{A_{m-3}^mD_{m-2}^m}{D_{m-2}^m A_{m-1}^m} &= \frac{[A_{m-3}^mA_{m-2}^mG_{m-2}]}{[A_{m-2}^mA_{m-1}^mG_{m-2}]}\\
\frac{A_{m-2}^mD_{m-1}^m}{D_{m-1}^m A_{m}^m} &= \frac{[A_{m-2}^mA_{m-1}^mG_{m-1}]}{[A_{m-1}^mA_{m}^mG_{m-1}]}\\
\frac{A_{m-1}^mD_{m}^m}{D_{m}^m A_{1}^m} &= \frac{[A_{m-1}^mA_{m}^mG_{m-1}]}{[A_{m}^mA_{1}^mG_{m-1}]}\\
\frac{A_{m}^mD_{1}^m}{D_{1}^m A_{2}^m} &= \frac{[A_{m}^mA_{1}^mG_{1}]}{[A_{1}^mA_{2}^mG_{1}]}
\end{align}
Also, note that:
\begin{align}
\frac{[A_{m-2}^mA_{m-1}^mG_{m-1}]}{[A_{m-2}^mA_{m-1}^mG_{m-2}]}&=\frac{A_{m-1}^{m-1}G_{m-1}}{A_{m-1}^{m-1}G_{m-2}}\\
\frac{[A_{m}^mA_{1}^mG_{1}]}{[A_{m}^mA_{1}^mG_{m-1}]}&=\frac{A_{m-1}^{m-1}G_{1}}{A_{m-1}^{m-1}G_{m-1}}
\end{align}
Now by putting all of this together, we get that the left-hand side of~(\ref{eq:cevafull3}) is equal to
\begin{equation}\label{eq:cevafull4}
\frac{[A_{m-3}^mA_{m-2}^mG_{m-2}]}{[A_{1}^mA_{2}^mG_{1}]}\cdot \frac{A_{m-1}^{m-1}G_1}{A_{m-1}^{m-1}G_{m-2}}.
\end{equation}

Let us do the same for the right-hand side. Rewriting the fractions gives us:
\begin{align}
\frac{A_{m-3}^{m-1}D_{m-2}^{m-1}}{D_{m-2}^{m-1} A_{m-1}^{m-1}} &= \frac{[A_{m-3}^{m-1}A_{m-2}^{m-1}G_{m-2}]}{[A_{m-2}^{m-1}A_{m-1}^{m-1}G_{m-2}]}\\
\frac{A_{m-2}^{m-1}D_{m-1}^{m-1}}{D_{m-1}^{m-1} A_{1}^{m-1}} &= \frac{[A_{m-2}^{m-1}A_{m-1}^{m-1}G_{m-2}]}{[A_{m-1}^{m-1}A_{1}^{m-1}G_{1}]}\cdot \frac{A_{m-1}^{m-1}G_1}{A_{m-1}^{m-1}G_{m-2}}\\
\frac{A_{m-1}^{m-1}D_{1}^{m-1}}{D_{1}^{m-1} A_{2}^{m-1}} &= \frac{[A_{m-1}^{m-1}A_{1}^{m-1}G_{1}]}{[A_{1}^{m-1}A_{2}^{m-1}G_{1}]}
\end{align}
Hence the right-hand side of~(\ref{eq:cevafull3}) is equal to
\begin{equation}\label{eq:cevafull5}
\frac{[A_{m-3}^{m-1}A_{m-2}^{m-1}G_{m-2}]}{[A_{1}^{m-1}A_{2}^{m-1}G_{1}]}\cdot \frac{A_{m-1}^{m-1}G_1}{A_{m-1}^{m-1}G_{m-2}},
\end{equation}
which is equal to~(\ref{eq:cevafull4}). We have proven now that in every step, the left-hand side of~(\ref{eq:cevafull}) is constant.\\

Now assume that there exists a sequence of $(n-3)$ steps which reduces $P^n$ to a triangle $P^3$ for which $g_1^3,g_2^3,g_3^3$ are concurrent. Then by Ceva,
\begin{equation}\label{eq:cevafull6}\frac{A_1^3D_2^3}{D_2^3 A_3^3}\cdot \frac{A_2^3D_3^3}{D_3^3 A_1^3}\cdot \frac{A_3^3D_1^3}{D_1^3 A_{2}^3}=1.\end{equation}
But as the product is constant, it must have been $1$ right from the start. This implies that also every other sequence of $(n-3)$ steps ends up with a triangle for which this product is still $1$, hence by Ceva the lines $g_1^3,g_2^3,g_3^3$ of every such possible triangle are concurrent. This completes the proof.
\end{proof}

We will call lines $g_1^n,\dots,g_n^n$ of an $n$-gon with the property as in Theorem~\ref{thm:cevafull} \textit{pseudo-concurrent}. Note that concurrency always implies pseudo-concurrency. Also, note that the order of the lines $g_1^n,\dots,g_n^n$ respectively of the vertices $A_1^n,\dots,A_n^n$ matters, i.e. if the lines $g_1^4,g_2^4,g_3^4,g_4^4$ of the quadrilateral $A_1^4A_2^4A_3^4A_4^4$ are pseudo-concurrent, then this does not imply that the lines $g_1^4,g_3^4,g_2^4,g_4^4$ of the quadrilateral $A_1^4A_3^4A_2^4A_4^4$ are pseudo-concurrent.\\

A rather trivial but still quite nice consequence of this definition is the following result:

\begin{theorem}\label{thm:inneranglebisectorspseudoconcurrent}
The internal angle bisectors of every $n$-gon are pseudo-concurrent.
\end{theorem}
\begin{proof}
Using the notation of Theorem~\ref{thm:cevafull}, this theorem follows by observing that whenever we replace some $g_i^m$ and $g_{i+1}^m$ by a new line $g_i^{m-1}$, this new line will thanks to its definition also be an internal angle bisector of the reduced $(m-1)$-gon $P^{m-1}$ (this is also true if $P^m$ is not convex). After $(n-3)$ steps, we get a triangle $P^3$ and lines $g_1^3,g_2^3,g_3^3$ which are the internal angle bisectors of this triangle and hence are concurrent. Thus, the internal angle bisectors of the original $n$-gon are pseudo-concurrent.
\end{proof}

It is possible to consider the external angle bisectors as well and get the following version:
\begin{theorem}\label{thm:anglebisectorspseudoconcurrent}
Let $A_1,\dots,A_n$ be an $n$-gon. For every $i\in\{1,\dots,n\}$, choose one of the two angle bisectors in $A_i$. If the number of chosen external angle bisectors is even, then the chosen angle bisectors are pseudo-concurrent.
\end{theorem}
\begin{proof}
This follows by thoroughly checking all cases when replacing $g_i^m$ and $g_{i+1}^m$ by $g_i^{m-1}$ and seeing that
\begin{itemize}
\item whenever $g_i^m$ and $g_{i+1}^m$ are both internal or both external angle bisectors, then $g_i^{m-1}$ is an internal angle bisector, and
\item whenever one of $g_i^m$ and $g_{i+1}^m$ is an internal and the other an external angle bisector, then $g_i^{m-1}$ is an external angle bisector,
\end{itemize}
hence the number of external angle bisectors will always stay even. In the end, we get a triangle $P^3$ and either three internal angle bisectors or one internal and two external angle bisectors, hence the remaining angle bisectors will be concurrent. From this, we get that the original angle bisectors are pseudo-concurrent.
\end{proof}

We can also generalize Menelaos in a similar way:

\begin{theorem}[Two-sided Menelaos' Theorem for $n$-gons]\label{thm:menelaosfull}
For $n\geq 3$, let $P^n=[A_1^n,\dots,A_n^n]$ be an $n$-gon and let $B_1^n\dots,B_n^n$ be points such that for each $i\in\{1,\dots,n\}$, $B_i^n$ lies on $A_i^n\times A_{i+1}^n$.
For $m\in\{4,\dots,n\}$, the $m$-gon $P^m$ can be reduced in a step to the $(m-1)$-gon $P^{m-1}$ by obeying the following procedure:
\begin{itemize}
\item Choose a vertex $A_i^m$, $i\in\{1,\dots,m\}$.
\item Replace $B_{i-1}^m$ and $B_{i}^m$ by\\
\[B_{i-1}^{m-1}=(A_{i-1}^m\times A_{i+1}^m)\times (B_{i-1}^m\times B_{i}^m).\]
\item Let
\[A_j^{m-1}=\begin{cases}A_j^m&\text{for $j\in\{1,\dots,i-1\}$,}\\A_{j+1}^m&\text{for $j\in\{i,\dots,m-1\}.$}\end{cases}\]
\item Let
\[B_j^{m-1}=\begin{cases}B_j^m&\text{for $j\in\{1,\dots,i-2\}$,}\\B_{j+1}^m&\text{for $j\in\{i,\dots,m-1\}.$}\end{cases}\]
\end{itemize}
If it is possible to reduce $P^n$ in $(n-3)$ steps to a triangle $P^3$ for which $B_1^3,B_2^3,B_3^3$ are collinear, then every sequence of $(n-3)$ steps will lead to such a triangle (see Figure~\ref{fig-menelaos-n}). Furthermore, this is the case if and only if
\begin{equation}\label{eq:menelaosfull}\frac{A_1^nB_1^n}{B_1^n A_2^n}\cdot \frac{A_2^nB_2^n}{B_2^n A_3^n}\cdot \ldots\cdot \frac{A_n^nB_n^n}{B_n^n A_{1}^n}=(-1)^n.\end{equation}
\end{theorem}
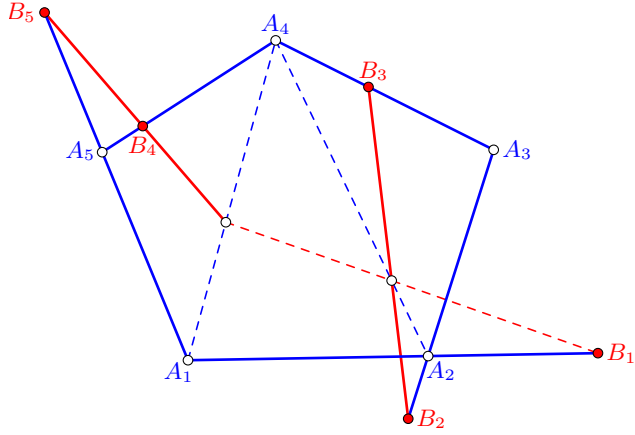
\begin{figure}[h]
\begin{center}
\begin{tikzpicture}[line cap=round,line join=round,x=15,y=15]
\clip(-0.4232218529662531,-4.867553649509905) rectangle (15.259281677451582,7.141214796956144);
\draw [line width=1pt,color=blue] (1.8940165005522343,2.6850491978933007)-- (6.222343895627236,5.4974525870834645);
\draw [line width=1pt,color=blue] (6.222343895627236,5.4974525870834645)-- (11.660131863649159,2.744508043330725);
\draw [line width=.6pt,dash pattern=on 3pt off 3pt,color=blue] (4.034534936299505,-2.5473292006000285)-- (6.222343895627236,5.4974525870834645);
\draw [line width=.6pt,dash pattern=on 3pt off 3pt,color=blue] (6.222343895627236,5.4974525870834645)-- (10.025013614119993,-2.443276221084536);
\draw [line width=.6pt,dash pattern=on 3pt off 3pt,color=red] (4.978492984891182,0.9236943937000124)-- (14.26146916136787,-2.369690144581471);
\draw [line width=1pt,color=red] (0.45648149954395967,6.199023644802416)-- (4.978492984891182,0.9236943937000124);
\draw [line width=1pt,color=blue] (9.52843345225824,-4.0187896437186446)-- (11.660131863649159,2.744508043330725);
\draw [line width=1pt,color=red] (9.52843345225824,-4.0187896437186446)-- (8.53677315906074,4.325745368975682);
\draw [line width=1pt,color=blue] (0.45648149954395967,6.199023644802416)-- (4.034534936299505,-2.5473292006000285);
\draw [line width=1pt,color=blue] (4.034534936299505,-2.5473292006000285)-- (14.26146916136787,-2.369690144581471);
\begin{small}
\draw [fill=white] (4.034534936299505,-2.5473292006000285) circle (1.8pt);
\draw[color=blue] (3.8,-2.85) node {$A_1$};
\draw [fill=white] (10.025013614119993,-2.443276221084536) circle (1.8pt);
\draw[color=blue] (10.35,-2.82) node {$A_2$};
\draw [fill=white] (11.660131863649159,2.744508043330725) circle (1.8pt);
\draw[color=blue] (11.66,2.74) node[anchor=west] {$A_3$};
\draw [fill=white] (6.222343895627236,5.4974525870834645) circle (1.8pt);
\draw[color=blue] (6.223078446763723,5.9) node {$A_4$};
\draw [fill=white] (1.8940165005522343,2.6850491978933007) circle (1.8pt);
\draw[color=blue] (1.35,2.7169316851002314) node {$A_5$};
\draw [fill=red] (14.26146916136787,-2.369690144581471) circle (1.8pt);
\draw[color=red] (14.26,-2.37) node[anchor=west] {$B_1$};
\draw [fill=white] (4.978492984891182,0.9236943937000124) circle (1.8pt);
\draw [fill=white] (9.11549327715157,-0.5440173119404466) circle (1.8pt);
\draw [fill=red] (0.45648149954395967,6.199023644802416) circle (1.8pt);
\draw[color=red] (.456,6.2) node[anchor=east] {$B_5$};
\draw [fill=red] (2.9053611534502966,3.3421873646332436) circle (1.8pt);
\draw[color=red] (2.904866112871789,3.3) node[anchor=north] {$B_4$};
\draw [fill=red] (8.53677315906074,4.325745368975682) circle (1.8pt);
\draw[color=red] (8.66,4.73) node {$B_3$};
\draw [fill=red] (9.52843345225824,-4.0187896437186446) circle (1.8pt);
\draw[color=red] (9.52,-4.02) node[anchor=west] {$B_2$};
\end{small}
\end{tikzpicture}
\caption{Theorem of Menelaos for the pentagon $A_1^5\ldots A_5^5$.}\label{fig-menelaos-n}
\end{center}
\end{figure}
\begin{proof}
We will show that in each step, the left-hand side of~(\ref{eq:menelaosfull}) will only change by a factor $-1$. Precisely, we have to show that
\begin{equation}\label{eq:menelaosfull2}
\frac{A_{m-1}^mB_{m-1}^m}{B_{m-1}^mA_m^m}\cdot\frac{A_m^mB_m^m}{B_m^mA_1^m}=-\frac{A_{m-1}^{m-1}B_{m-1}^{m-1}}{B_{m-1}^{m-1}A_1^{m-1}}
\end{equation}
if the vertex $A_m^m$ was chosen when reducing $P^m$ to $P^{m-1}$. But this follows directly by applying Menelaos' 
Theorem to the triangle $A_{m-1}^mA_m^mA_1^m$ (see Figure~\ref{fig-menelaos-proof}).
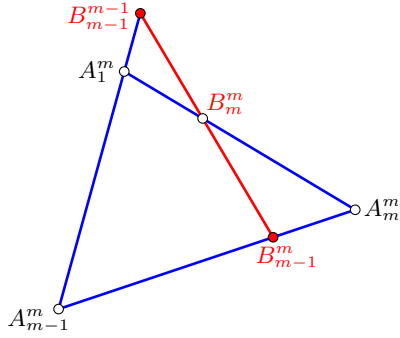
\begin{figure}[h]
\begin{center}
\begin{tikzpicture}[line cap=round,line join=round,x=19,y=15]
\clip(-1.8,-2.235170262053546) rectangle (6.071000392531282,6.505280017247809);
\draw [line width=1pt,color=blue] (-0.6032550078195817,-1.6554465190386645)-- (5.238576556407345,0.8418249893331514);
\draw [line width=1pt,color=blue] (5.238576556407345,0.8418249893331514)-- (0.6899748804443944,4.320167447422467);
\draw [line width=1pt,color=red] (1.0067518811250238,5.783895657463995)-- (3.6215219668120864,0.15056501210159023);
\draw [line width=1pt,color=blue] (1.0067518811250238,5.783895657463995)-- (-0.6032550078195817,-1.6554465190386645);
\begin{small}
\draw [fill=white] (-0.6032550078195817,-1.6554465190386645) circle (1.8pt);
\draw[color=black] (-0.99,-1.95) node {$A_{m-1}^m$};
\draw [fill=white] (5.238576556407345,0.8418249893331514) circle (1.8pt);
\draw[color=black] (5.238,.84) node[anchor=west] {$A_m^m$};
\draw [fill=white] (0.6899748804443944,4.320167447422467) circle (1.8pt);
\draw[color=black] (0.69,4.32) node[anchor=east] {$A_1^m$};
\draw [fill=white] (2.2343145324361817,3.1392018311934526) circle (1.8pt);
\draw[color=red] (2.13,3.03) node[anchor=south west] {$B_{m}^m$};
\draw [fill=red] (3.6215219668120864,0.15056501210159023) circle (1.8pt);
\draw[color=red] (3.9,.15) node[anchor=north] {$B_{m-1}^m$};
\draw [fill=red] (1.0067518811250238,5.783895657463995) circle (1.8pt);
\draw[color=red] (1,5.7) node[anchor=east] {$B_{m-1}^{m-1}$};
\end{small}
\end{tikzpicture}
\caption{Proof of Theorem~\ref{thm:menelaosfull}.}\label{fig-menelaos-proof}
\end{center}
\end{figure}

Hence, if there exists a sequence of $(n-3)$ steps that reduces $P^n$ to a triangle $P^3$ for which $B_1^3,B_2^3,B_3^3$ are collinear, then by Menelaos' Theorem,
\begin{equation}\label{eq:menelaosfull3}
\frac{A_1^3B_1^3}{B_1^3A_2^3}\cdot\frac{A_2^3B_2^3}{B_2^3A_3^3}\cdot\frac{A_3^3B_3^3}{B_3^3A_1^3}=-1.
\end{equation}
But as the product only changed by a factor $-1$ in each of the $(n-3)$ steps, we get that
\begin{equation}\label{eq:menelaosfull4}
\frac{A_1^nB_1^n}{B_1^n A_2^n}\cdot \frac{A_2^nB_2^n}{B_2^n A_3^n}\cdot \ldots\cdot \frac{A_n^nB_n^n}{B_n^n A_{1}^n}=(-1)\cdot (-1)^{n-3}=(-1)^n.\end{equation}
This implies that also every other sequence of $(n-3)$ steps leads to a triangle for which $B_1^3,B_2^3,B_3^3$ are collinear. This completes the proof.
\end{proof}

As before, we suggest to call points $B_1^n,\dots,B_n^n$ of an $n$-gon with the property as in Theorem~\ref{thm:menelaosfull} \textit{pseudo-collinear}. Again, collinearity implies pseudo-collinearity, and the order of the points matters.

An immediate consequence of these new definitions is the following result:

\begin{corollary}\label{cor:dualconfig}
If the points $B_1,\dots,B_n$ on the sides of an $n$-gon are pseudo-collinear, then the lines $b_1,\dots,b_n$ in the dual configuration are pseudo-concurrent and vice-versa.
\end{corollary}

The new concepts of \textit{pseudo-concurrency} and \textit{pseudo-collinearity} are a very natural way to generalize the well-known terms \textit{concurrency} and \textit{collinearity}. As Theorem~\ref{thm:inneranglebisectorspseudoconcurrent} and Corollary~\ref{cor:dualconfig} show, it's easy to transfer some basic results to these new concepts. We are convinced that many more results can be carried over to this new setting, and are looking forward to whatever results this new perception may inspire.

\bibliographystyle{plain}
%\nocite{*}
%\bibliography{CommunicatingHarmonicPencils}

\begin{thebibliography}{10}

\bibitem{coxeter}
H.~S.~M. Coxeter.
\newblock {\em Introduction to geometry}.
\newblock Wiley Classics Library. John Wiley \& Sons, Inc., New York, 1989.
\newblock Reprint of the 1969 edition.

\bibitem{gauss}
Carl~Friedrich Gau\ss.
\newblock {Bestimmung der gr\"o\ss ten Ellipse, welche die vier Seiten eines
  gegebenen Vierecks ber\"uhrt}.
\newblock {\em {Zach's Monatliche Correspondenz f\"ur Erd- und Himmelskunde}},
  22:112--121, 1810.

\bibitem{werke}
Carl~Friedrich Gau\ss.
\newblock {\em Werke}, volume~4.
\newblock 1880.

\bibitem{gruenbaum95}
Branko Gr\"unbaum and Geoffrey~C. Shephard.
\newblock Ceva, {M}enelaus, and the area principle.
\newblock {\em Math. Mag.}, 68(4):254--268, 1995.

\bibitem{gruenbaum97}
Branko Gr\"unbaum and Geoffrey~C. Shephard.
\newblock Ceva, {M}enelaus, and selftransversality.
\newblock {\em Geom. Dedicata}, 65(2):179--192, 1997.

\bibitem{gudermann}
Christoph Gudermann.
\newblock {\em {Grundriss der analytischen Sph{\"a}rik}}.
\newblock Du Mont-Schauberg, 1830.

\bibitem{hh}
Lorenz Halbeisen and Norbert Hungerb\"uhler.
\newblock Conjugate conics and closed chains of {P}oncelet polygons.
\newblock {\em Mitt. Math. Ges. Hamburg}, 36:5--28, 2016.

\bibitem{miquel}
Auguste Miquel.
\newblock {Th\'eor\`emes de G\'eom\'etrie}.
\newblock {\em {Journal de math\'ematiques pures et appliqu\'ees 1re s\'erie}},
  3:485--487, 1838.

\bibitem{moss}
Thomas Moss.
\newblock Exercise 396.
\newblock {\em {Lady's Diary or Woman's Almanach}}, 1755.

\bibitem{newton}
Isaak Newton.
\newblock {\em Philosophiae naturalis principia mathematica}.
\newblock J.~Societatis Regiae ac Typis J.~Streater, 1686.

\bibitem{steiner}
Jakob Steiner.
\newblock {Questions propos\'ees. Th\'eor\`eme sur le quadrilat\`ere complet}.
\newblock {\em Annales de Gergonne}, 18:302--304, 1827--1828.

\end{thebibliography}

\end{document}